\documentclass[11pt,a4paper]{article}
\usepackage[T1]{fontenc}
\usepackage[utf8]{inputenc}
\usepackage{lmodern}
\usepackage{amsmath,amssymb,amsthm,mathtools,mathrsfs}
\usepackage{booktabs,tabularx,array,enumitem}
\usepackage{microtype,xurl}
\usepackage[margin=0.88in]{geometry}
\usepackage{xcolor,cite,caption,needspace}
\definecolor{linkblue}{RGB}{26,52,108}
\usepackage[colorlinks=true,linkcolor=linkblue,citecolor=linkblue,urlcolor=linkblue]{hyperref}
\usepackage[nameinlink,noabbrev]{cleveref}
\hypersetup{pdftitle={Gradient--Hessian Query Tradeoffs: Matching Lower Bounds in Logarithmic Dimension},pdfauthor={},pdfsubject={Randomized exact-oracle lower bounds for nonconvex stationarity},pdfkeywords={nonconvex optimization, oracle complexity, gradient queries, Hessian queries, lower bounds}}
\allowdisplaybreaks[1]
\newtheorem{theorem}{Theorem}[section]
\newtheorem{lemma}[theorem]{Lemma}
\newtheorem{proposition}[theorem]{Proposition}
\newtheorem{corollary}[theorem]{Corollary}
\theoremstyle{definition}

\theoremstyle{remark}

\newcolumntype{Y}{>{\raggedright\arraybackslash}X}
\newcommand{\R}{\mathbb R}\newcommand{\N}{\mathbb N}
\newcommand{\E}{\mathbb E}\newcommand{\PP}{\mathbb P}
\newcommand{\eps}{\varepsilon}
\newcommand{\norm}[1]{\left\lVert#1\right\rVert}
\newcommand{\ip}[2]{\langle #1,#2\rangle}
\newcommand{\op}{\mathrm{op}}\newcommand{\Frob}{\mathrm F}

\DeclareMathOperator{\Lip}{Lip}
\DeclareMathOperator{\tr}{tr}\DeclareMathOperator{\Var}{Var}
\DeclareMathOperator{\Ent}{Ent}\DeclareMathOperator{\spec}{spec}
\DeclareMathOperator{\diag}{diag}\DeclareMathOperator{\range}{range}
\title{\textbf{Gradient--Hessian Query Tradeoffs:\\Matching Lower Bounds in Logarithmic Dimension}}
\newcommand{\authornames}{Wendao Wu, Haihan Zhang, Chenheng Zhang, Yanyi Li, Chunyuan Zheng, Cong Fang, Haoxuan Li, Zhouchen Lin}
\newcommand{\affiliation}{Peking University}
\newcommand{\contactemail}{colombo.2026@outlook.com}
\newcommand{\mail}[1]{\href{mailto:#1}{\texttt{#1}}}
\newif\ifanonymous
\anonymousfalse
\ifanonymous
\author{Anonymous Authors}
\hypersetup{pdfauthor={Anonymous Authors}}
\else
\author{Wendao Wu$^{1,*}$\quad Haihan Zhang$^{1,*}$\quad Chenheng Zhang$^{1,*}$\\[0.20em]
Yanyi Li$^{1,*}$\quad Chunyuan Zheng$^1$\\[0.20em]
Cong Fang$^{1,\dagger}$\quad Haoxuan Li$^{1,\dagger}$\quad Zhouchen Lin$^{1,\dagger}$\\[0.60em]
\small $^1$\affiliation\\[0.35em]
\footnotesize\mail{\contactemail}\\[0.35em]
\small $^*$Equal contribution.\quad $^\dagger$Corresponding authors.}
\hypersetup{pdfauthor={\authornames}}
\fi
\date{}

\begin{document}
\maketitle
\begin{abstract}
We establish matching lower bounds for the tradeoff between gradient and
Hessian queries in nonconvex optimization with Lipschitz Hessian. In
normalized units, the initial objective gap is at most $T$, the Hessian
Lipschitz constant is at most one, and the target gradient norm is one.
For sufficiently large $T$, dimension $d\ge C\log(2+T)$, and a prescribed
Hessian budget $1\le h$, $h+1\le aT$, every randomized algorithm requires
\[
 \Omega\!\left(\min\left\{Td^{1/3},
                 \frac{T^{3/2}}{\sqrt{h+1}}\right\}\right)
\]
ordinary queries. An ordinary query returns the value and gradient; a
Hessian query returns the complete second jet and is charged only to its
own budget. Both budgets are worst-case caps. Queries and outputs may be
arbitrary real points, and no Hessian-norm bound is imposed. The hard input
law consists of fixed, globally $C^3$, coercive functions. The construction
combines sequential readouts of a continuous random matrix with localized
completion blocks and reliability dimensions that vary across hierarchy
levels. Premature disclosures are charged by their expected task cost,
while every oracle reply remains the exact jet of the presampled objective.
This yields a logarithmic dimension threshold and a depth-independent
third-derivative bound. Together with a translation of the NALEN\cite{CLLZ26} upper
bound, the result determines the query tradeoff up to absolute constants
throughout the stated range.
\end{abstract}

{\small
\begin{samepage}
\noindent\textbf{AI Usage.}
Nearly the entire research pipeline for this paper was carried out by
Colombo, our laboratory's internal auto-research system, powered by
GPT-5.6 Sol. The authors have reviewed and approved the
mathematical claims and presentation, and take responsibility for the
final manuscript. The system's technical report will be made public at
a later date.
\par
\end{samepage}
}

\paragraph{Publication Intent and Recognition of Independent Work.}

This preprint is released as a research output of our study of
automated scientific research, for public scrutiny and discussion.
As we do not intend to submit this manuscript to a journal or
conference, it has not undergone the full editorial refinement
and internal review we would normally require for submission.
The authors remain responsible for its contents and for correcting
any errors or inaccurate attribution.

We recognize that releasing these results may affect researchers
who have independently been pursuing the same questions, including
work that is not yet public.
If this applies to your work, we invite you to contact us at
\texttt{colombo.2026@outlook.com}.
A brief description of the overlap, together with any dated
materials you are comfortable sharing, would help us understand
its scope and chronology.
Such materials might include draft excerpts, presentation slides,
or revision records; a complete manuscript or unpublished proof
is not required to begin the conversation.

We will review these cases carefully and update citations,
acknowledgments, and accounts of independent or concurrent work
as appropriate, including revising priority claims where warranted.
Nonpublic materials will be treated confidentially and used only
to assess the overlap unless their authors authorize otherwise.
We will obtain explicit permission from the authors before
publicly discussing such materials or providing them to an
automated research system.

\section{Introduction}\label{sec:introduction}

Finding a point with a small gradient is a basic oracle problem in
nonconvex optimization. When the Hessian is Lipschitz, second-order
information can reduce the number of gradient evaluations, but obtaining
that information may be substantially more expensive. This leads to a
two-budget question: how many ordinary queries remain necessary when an
algorithm is allowed only a prescribed number of Hessian queries?

\Needspace{8\baselineskip}
Let $f:\R^d\to\R$ be bounded below, with
$f(x_0)-\inf f\le\Delta$ and $L_2$-Lipschitz Hessian. The target is
$\norm{\nabla f(x)}\le\eps$, and the dimensionless parameter
\[
 T=\frac{\Delta\sqrt{L_2}}{\eps^{3/2}}
\]
is the natural second-order complexity scale. Translation and rescaling
reduce the problem to $x_0=0$, $L_2=1$, and $\eps=1$, with initial
gap at most $T$. An ordinary query returns the value and gradient;
a rich query returns the value, gradient, and Hessian together.
We count these query types separately, so a rich query consumes no
ordinary query. Write $Q_g^*(h;d,T)$ for the least worst-case ordinary
budget that suffices with at most $h$ rich queries and success probability
at least $2/3$ on every fixed input.

Classical second-order theory identifies the linear scale $T$ when
Hessian access is unrestricted~\cite{NP06,CDHS20}. Recent upper bounds
describe how gradient complexity increases as Hessian access becomes
scarce~\cite{adil,CLLZ26}. In particular, Chen, Liu, Luo, and
Zhang~\cite{CLLZ26} remove the logarithmic overhead in the nonconvex
upper bound using lazy Hessian updates. A total-query lower bound alone
does not determine this tradeoff: it does not identify which queries
must remain ordinary when the rich budget is prescribed. Moreover, a
construction whose dimension grows rapidly with $T$ leaves open whether
the same obstruction persists in substantially fewer coordinates.

\paragraph{Our result.}
We prove that there are absolute constants $C,c,T_*>0$ and $0<a<1$
such that
\begin{equation}\label{eq:intro-main}
 Q_g^*(h;d,T)\ge
 c\min\left\{Td^{1/3},\frac{T^{3/2}}{\sqrt{h+1}}\right\}
\end{equation}
whenever $T\ge T_*$, $d\ge C\log(2+T)$, and
$1\le h$, $h+1\le aT$. The lower bound applies to randomized,
adaptive algorithms with arbitrary exact-real queries and arbitrary
outputs, including points that were never queried. There is no span,
radius, precision, or Hessian-norm restriction. The hard input
distribution is sampled before the private seed and consists of globally
$C^3$, coercive objectives. Thus the separate-budget lower tradeoff holds
already in logarithmic dimension.

\begin{table}[!t]
\caption{Previous upper bounds and the matching lower bounds of this work.
The two tradeoff panels compare bounds in the same dimension regime;
each bold row matches the Chen et~al. row directly above it.}
\label{tab:comparison}
\centering
\small
\setlength{\tabcolsep}{3pt}
\renewcommand{\arraystretch}{1.20}
\begin{tabularx}{\textwidth}{@{}
 >{\centering\arraybackslash}m{0.20\textwidth}
 >{\centering\arraybackslash}m{0.095\textwidth}
 >{\centering\arraybackslash}m{0.095\textwidth}
 >{\centering\arraybackslash}m{0.19\textwidth}
 >{\centering\arraybackslash}m{0.20\textwidth}
 >{\centering\arraybackslash}X@{}}
\toprule
Work & Oracle & \shortstack{Algorithm\\class}
     & Upper bound & Lower bound & \shortstack{Budget /\\dimension}\\
\midrule
\multicolumn{6}{@{}l}{\textit{Total-query benchmarks}}\\[4pt]
\shortstack{Nesterov--Polyak\\(2006)~\cite{NP06}}
& \shortstack{Second\\order} & Det.
& $\begin{gathered}G=0,\\H=O(T)\end{gathered}$
& --- & All $d$\\[7pt]
\shortstack{Carmon et~al.\\(2020)~\cite{CDHS20}}
& \shortstack{Second\\order} & Rand.
& ---
& $G+H=\Omega(T)$
& $d=O(T^2\mathcal L)$\newline suffices\\[4pt]
\midrule
\multicolumn{6}{@{}l}{\textit{Dimension-dependent regime:}
 $\boldsymbol{d\le(T/h)^{3/2}}$}\\[4pt]
\shortstack{Chen et~al.\\(2026)~\cite{CLLZ26}$^\dagger$}
& Gradient & Det.
& $G=O(Td^{1/3})$
& --- & $H=0$\\[7pt]
\shortstack{\textbf{This work}\\\textbf{(Theorem~\ref{thm:main})}}
& \textbf{Mixed} & \textbf{Rand.}
& --- & {\boldmath$G=\Omega(Td^{1/3})$}
& {\boldmath$H\le h$}\\[4pt]
\midrule
\multicolumn{6}{@{}l}{\textit{Hessian-budget-dependent regime:}
 $\boldsymbol{d\ge(T/h)^{3/2}}$}\\[4pt]
\shortstack{Adil et~al.\\(2025)~\cite{adil}}
& Mixed & Det.
& $G=\widetilde O\!\left(\frac{T^{3/2}}{\sqrt h}\right)$
& ---
& $H\le h$\\[7pt]
\shortstack{Chen et~al.\\(2026)~\cite{CLLZ26}$^\dagger$}
& Mixed & Det.
& $G=O\!\left(\frac{T^{3/2}}{\sqrt h}\right)$
& ---
& $H\le h$\\[7pt]
\shortstack{\textbf{This work}\\\textbf{(Theorem~\ref{thm:main})}}
& \textbf{Mixed} & \textbf{Rand.}
& --- & {\boldmath$G=\Omega\!\left(\frac{T^{3/2}}{\sqrt h}\right)$}
& {\boldmath$H\le h$}\\
\bottomrule
\end{tabularx}
\par\vspace{6pt}
\begin{minipage}{\textwidth}
\footnotesize
\textit{Common model and counts.}
All queries are exact and may be made anywhere in $\R^d$.
The Hessian is one-Lipschitz, the initial gap is at most $T$,
and the target gradient norm is one. No Hessian-norm bound is assumed.
$G$ counts ordinary value--gradient queries; $H$ counts joint
value--gradient--Hessian queries. A rich query is charged only to $H$.
Both are hard caps.
\par\smallskip
\textit{Comparison range and classes.}
Both tradeoff panels assume $T\ge T_*$, $d\ge C\mathcal L$,
$h\ge1$, and $h+1\le aT$, where $\mathcal L=\log(2+T)$ and $0<a<1$.
The total-query benchmarks require $T\ge1$; the Carmon et~al.
dimension is sufficient for their hard construction.
``Det.'' denotes a deterministic upper bound; ``Rand.'' means the lower
bound covers randomized algorithms. ``Mixed'' separates the two query
budgets; ``Gradient'' uses no Hessian calls.
\par\smallskip
$\dagger$~The two Chen et~al. rows are the mixed-budget and finite-difference
adaptations proved in \Cref{prop:upper-translation}.
In the dimension-dependent panel, $O(d+Td^{1/3})=O(Td^{1/3})$
because $d^{2/3}\le T/h\le T$; a gradient-only method is admissible
under the allowance $H\le h$.
Replacing $h+1$ by $h$ in the lower bound changes only an absolute constant.
Adil et~al.'s $\widetilde O$ hides polylogarithmic factors in $2+hT^2$.
\end{minipage}
\end{table}

\paragraph{How the lower bound matches previous work.}
The two terms in \Cref{eq:intro-main} correspond to two ways of using
the Chen et~al. upper bound~\cite{CLLZ26}, translated to our oracle
model in \Cref{prop:upper-translation}. The first uses only gradients
and costs $O(d+Td^{1/3})$; the second uses at most $h$ rich queries
and costs $O(T^{3/2}/\sqrt h)$. An algorithm with rich-query budget $h$
may choose either method, including one that uses no rich queries.

When $d\le(T/h)^{3/2}$, the dimension term is the smaller one.
Here $d^{2/3}\le T/h\le T$, so the initialization cost satisfies
$d\le Td^{1/3}$. The upper bound is therefore $O(Td^{1/3})$.
Our theorem gives $\Omega(Td^{1/3})$ in this same regime, even when
the algorithm is allowed $h$ rich queries. These are the adjacent upper
and lower rows in the dimension-dependent panel of \Cref{tab:comparison}.

When $d\ge(T/h)^{3/2}$, the mixed-query method gives the smaller rate,
$O(T^{3/2}/\sqrt h)$. Our theorem gives
$\Omega(T^{3/2}/\sqrt h)$ in that regime, as shown in the final two
rows of \Cref{tab:comparison}. Adil et~al.~\cite{adil} obtain the same
polynomial rate with a logarithmic overhead. For both comparisons,
$h\le h+1\le2h$ makes the $h+1$ in the exact lower-bound formula
equivalent to $h$ up to an absolute constant.

Choosing the better upper-bound method yields, for $1\le h\le T$,
\begin{equation}\label{eq:upper-envelope}
 Q_g^*(h;d,T)\le
 O\!\left(\min\left\{Td^{1/3},
                   \frac{T^{3/2}}{\sqrt h}\right\}\right).
\end{equation}
Thus \Cref{eq:intro-main,eq:upper-envelope} match up to absolute
constant factors throughout the range of \Cref{thm:main}.
The two bold table rows are consequences of this one lower bound
under their respective dimension conditions; they are not two
unconditional lower bounds. The total-query endpoint results at the
top of the table provide context, whereas the Chen et~al. rates provide
the matching upper bounds. The dimension threshold still grows with $T$,
so this comparison does not settle the fixed-dimensional regime.

\paragraph{Proof ingredients.}
A continuous rectangular-Wishart matrix supplies $M$ sequential marked
tasks within each numerical leaf. Every earlier polynomial readout of
the same matrix is synthesized and charged to the product budget.
An adaptive posterior estimate then limits the probability of completing
a new task before its ordinary-query lifetime. Smooth registers embed
these tasks into a globally defined scalar whose small-gradient points
certify completion.

The logarithmic dimension bound comes from the completion hierarchy.
At higher levels, hidden data enter through perturbations supported in
small balls around random centers, together with completion ports.
The ratio between the center radius and support width grows with height.
Only the first $O(\log(2+M))$ levels need register dimension
$O(\log(2+M))$; all subsequent levels use constant-dimensional
registers. Separate coordinates at different heights preserve the
activation argument, while the increasing physical scales keep the
third-derivative bound independent of depth.

To analyze unrestricted queries, we allow premature discoveries and
charge the tasks they disclose. A comparison oracle processes these
discoveries before returning the literal value, gradient, or full second
jet of the fixed objective. Local visibility windows and stopped tests
bound the expected number of tasks lost to such repairs, without requiring
every hidden block to remain undiscovered throughout the run. Combining
this charge with the numerical lifetimes gives a single task ledger for
both query budgets. For $n$ leaves, the resulting construction has
dimension $O(M^{3/2}+\log(2+n))$, gap $O(nM)$, and ordinary lower bound
$\Omega(nM^{3/2})$ whenever $h\le n/2$. Optimizing $M$ and $n$ gives
\Cref{eq:intro-main}.

\paragraph{Organization.}
\Cref{sec:related} discusses related work, and \Cref{sec:model}
states the model and the parameter theorem.
\Cref{sec:gaussian,sec:posterior,sec:matrix} establish the probabilistic
and matrix ingredients. \Cref{sec:leaf,sec:localized,sec:first-height,sec:hierarchy-log}
construct the objective, and \Cref{sec:regularity-log} proves its
depth-uniform regularity. \Cref{sec:repairs,sec:repair-cost,sec:ledger-log}
analyze the comparison oracle and query budgets; \Cref{sec:finish-log}
completes the parameter optimization. \Cref{sec:scope} discusses the
remaining regimes. The appendices justify the upper-bound comparison
and examine an obstruction to identifying adjacent hierarchy levels.
\section{Related Work}\label{sec:related}

\paragraph{Stationarity with second-order oracles.}
Nesterov and Polyak~\cite{NP06} established the $O(T)$ iteration scale
for cubic regularization under a Lipschitz-Hessian assumption. In our
joint-rich oracle convention, the method uses $O(T)$ rich queries,
including a final stationarity check, and no ordinary queries.
Carmon, Duchi, Hinder, and
Sidford~\cite[Theorem~2]{CDHS20} prove a randomized $\Omega(T)$ lower
bound even with access to all derivatives. Their randomized construction
can use dimension $O(T^2\log(2+T))$. These results identify the total
second-order scale. A constraint on $G+H$, however, does not yield the
dependence on $G$ in \Cref{eq:intro-main} when $H$ is prescribed.
The arbitrary-output convention changes the cited query lower bound by
at most one additional tested point.

\paragraph{Balancing gradient and Hessian queries.}
Doikov, Chayti, and Jaggi~\cite{DCJ23} study the cost reduction obtained
by reusing Hessians over several iterations. Adil, Bullins, Sidford,
and Zhang~\cite{adil} obtain the mixed-query upper bound in
\Cref{tab:comparison}; their Corollary~2 and Appendix~G handle approximate
Hessians without a Hessian-norm bound, including the logarithmic factors.
Chen, Liu, Luo, and Zhang~\cite[Theorem~4.1]{CLLZ26} subsequently give
the NALEN method. With Hessian refresh period $m$, its iteration bound
in our normalization is $O(m+Tm^{1/3})$, with a constant number of gradient
calls per iteration and one Hessian call per period. In their weighted
cost model, this yields $O(\bar d+T\bar d^{1/3})$ equivalent gradient
calls, where $\bar d=c_H/c_g$ is the Hessian-to-gradient cost ratio.
This removes the logarithmic overhead in the preceding upper bound.
Equivalent gradient cost is distinct from an algorithm making only
ordinary queries. \Cref{app:upper-translation} derives both a prescribed
Hessian cap and a finite-difference implementation using only gradients;
these are the two Chen et~al. entries in \Cref{tab:comparison}.
Our theorem gives the matching randomized lower tradeoff for
$d\ge C\log(2+T)$ and the stated nontrivial Hessian-budget range.

\paragraph{Convex optimization with lazy second-order oracles.}
Zhang, Chen, Liu, and Zhang~\cite{ZCLZ26} establish nearly matching
bounds for convex optimization with lazy second-order oracles, using
function-value accuracy and an initial-distance bound. Their deterministic
lower bounds address a different optimization criterion from the
randomized nonconvex stationarity problem considered here. Accordingly,
those rates are not included in the common stationarity comparison of
\Cref{tab:comparison}.

\paragraph{First-order lower bounds under additional smoothness.}
Carmon et~al.~\cite[Theorem~2]{CDHS21} also prove lower bounds for
deterministic first-order methods when both the gradient and Hessian
Lipschitz constants are controlled. With
$\kappa_1=L_1/\sqrt{L_2\eps}\ge1$, their $p=2$ bound is
$\Omega(T\kappa_1^{3/7})$. It concerns an additional smoothness parameter
and a different algorithm class. We therefore do not treat it as a
matching lower bound for randomized mixed-oracle algorithms on the class
used here. The present theorem prescribes no global $L_1$ bound.

\paragraph{Adaptive matrix information and smooth embeddings.}
The rectangular-Wishart deferred-decision argument of Zhang et~al.
\cite[Lemma~4.1 and Theorem~4.6]{matvec} is a technical antecedent of
our matrix source. Their posterior overlap estimate concerns a rank-$k$
output projector. \Cref{lem:deferred,lem:overlap} give the factorization
and an arbitrary-output-rank version needed here, with complete proofs.
The chronological fork charges all earlier polynomial readouts of one
matrix. The optimization construction then embeds the sequential matrix
tasks into a fixed smooth scalar and accounts jointly for ordinary and
rich queries. Localized completion blocks, a reliability dimension that
varies by hierarchy level, and a charge for prematurely disclosed tasks
control the dimension of this embedding. Matrix approximation is a
different oracle task, so its query rates are not listed as stationarity
bounds in \Cref{tab:comparison}.
\section{Oracle Model and Main Results}\label{sec:model}
For integers $d\ge1$ and reals $T\ge0$, define
\[
 \mathcal F_d(T)=\{f\in C^2(\R^d):\inf_{\R^d}f> -\infty,
 \ f(0)-\inf f\le T,\ \Lip(\nabla^2f)\le1\}.
\]
All vector norms are Euclidean and matrix norms are operator norms.  An
ordinary query at $x$ returns $(f(x),\nabla f(x))$.  A rich query returns the
complete second jet $(f(x),\nabla f(x),\nabla^2f(x))$ and is charged only to
its rich budget.  Both query caps hold pathwise on every execution.  The
algorithm may be randomized and adaptive; its private seed is independent of
the fixed input.  Queries and the final output are arbitrary exact-real
vectors, and the output need not have been queried. Policies, query
selections, and outputs are measurable. Internal computation on returned
data is uncharged and is represented by a sequence of measurable steps;
its measurable termination time is finite on every legal input, without
a uniform running-time bound. Finite parameter-dependent advice independent
of the unknown input is allowed. A deterministic computation cutoff makes
the counterfactual simulations in the proof total; \Cref{sec:ledger-log}
removes this cutoff under the legal input law. Both original query caps are also enforced on counterfactual
histories, with termination before any excess query.

Let $Q_g^*(h;d,T)$ be the least ordinary-query cap of an algorithm which,
for every fixed $f\in\mathcal F_d(T)$, terminates and returns $x$ satisfying
$\norm{\nabla f(x)}\le1$ with probability at least $2/3$ while using at most
$h$ rich queries.

\begin{theorem}[Main lower bound]\label{thm:main}
There are absolute constants $C,a,c,T_*>0$ such that for every real
$T\ge T_*$, every integer $d\ge C\log(2+T)$, and every integer $h\ge1$ with
$h+1\le aT$,
\begin{equation}\label{eq:main}
 \boxed{\quad Q_g^*(h;d,T)\ge
 c\min\left\{Td^{1/3},\frac{T^{3/2}}{\sqrt{h+1}}\right\}.\quad}
\end{equation}
The hard input distribution is sampled before the algorithm's private seed.
The theorem applies to unrestricted randomized algorithms, not only to
zero-respecting, span-restricted, or deterministic algorithms.
\end{theorem}

We take $a<1$, decreasing it if necessary.
The proof proceeds through the following parameter theorem.  The symbol $L$
below denotes a finite hierarchy depth and is unrelated to a Hessian
Lipschitz coefficient.

\begin{theorem}[Uniform parameter construction]\label{thm:parameter}
There are absolute constants $C_E,C_d,c_0,J_0>0$ with the following
property.  Let $J\ge J_0$ be an integer, put
\[
 M=1+4096J^2,\qquad n=8^L\quad(L\in\N,\ L\ge1).
\]
If
\begin{equation}\label{eq:parameters}
 T\ge C_E nM,\qquad d\ge C_d\{M^{3/2}+\log(2+n)\},
\end{equation}
then for every integer $0\le h\le n/2$,
\begin{equation}\label{eq:parameter-lower}
 Q_g^*(h;d,T)\ge c_0nM^{3/2}.
\end{equation}
All constants are independent of $L,n,M,d,T,h$.
\end{theorem}

Every leaf below contains $M$ sequentially disclosed marked tasks sharing
one continuous random matrix.  A normally completed marked task uses
$\ell\asymp\sqrt M$ ordinary slots.  A rich query may retire its current
leaf, but no unretired sibling leaf.  Prematurely discovered subtrees are
not forbidden; their marked tasks are charged to a repair variable $Z$.
The construction will imply the exact task ledger
\begin{equation}\label{eq:road-ledger}
 nM\le hM+(G+1)/\ell+E_{\rm early}+Z,
\end{equation}
where $\E E_{\rm early}\le nM/256$ and $\E Z\le nM/256$.  This is a
single-family joint mixed-oracle lower bound, not a combination of separate
first- and second-order endpoint instances.

\subsection{Conventions}
Symbols $C,c$, with subscripts when useful, denote positive absolute
constants whose values may change between displays.  Dependence on
$n,M,L,d,T,h$ is always displayed.  $D^qF$ denotes the $q$-linear derivative
tensor with Euclidean operator norm; order zero means absolute value.
$J^2F$ denotes value, gradient, and Hessian together.  A gate is called
flat at a point when its complete derivative jet of the indicated order is
exactly zero, not merely small.  Every objective used in the argument is a
finite explicit formula; no infinite-dimensional limit is used.
\section{Gaussian Preliminaries}\label{sec:gaussian}

\subsection{Integration by parts, concentration, and variance}
Let \(\gamma_m\) be standard Gaussian measure on \(\R^m\).
Integration by parts against its density gives
\begin{equation}\label{eq:gibp}
\E[Z_i u(Z)]=\E[\partial_i u(Z)]\qquad(Z\sim\gamma_m)
\end{equation}
for smooth functions of appropriate growth. Boundary terms vanish because
of the Gaussian density; truncation extends the identity when necessary.
For a centered Gaussian vector with covariance \(K\), a linear change of
variables gives \(\E Z_i u(Z)=\sum_jK_{ij}\E\partial_j u(Z)\).

\begin{lemma}[Gaussian concentration and Poincare inequality]\label[lemma]{lem:gauss-conc}
If \(f:\R^m\to\R\) is \(b\)-Lipschitz, then
\[
\E e^{\lambda(f(Z)-\E f(Z))}\le e^{\lambda^2b^2/2}
\quad(\lambda\in\R),\qquad
\PP\{|f(Z)-\E f(Z)|\ge t\}\le2e^{-t^2/(2b^2)}.
\]
Also \(\Var(f(Z))\le\E\norm{\nabla f(Z)}^2\) when the right side is defined
by smooth approximation. Consequently
\[
\sqrt{m-1}\le\E\norm Z\le\sqrt m.
\]
The lower bound is interpreted as zero when \(m=1\).
\end{lemma}
\begin{proof}
We provide the entropy derivation. For a bounded smooth positive function
\(g\) bounded away from zero, set
\[
P_tg(x)=\E g(e^{-t}x+\sqrt{1-e^{-2t}}Z).
\]
Direct differentiation and \eqref{eq:gibp} show that its generator is
\(\mathscr L=\Delta-x\cdot\nabla\), that \(\gamma_m\) is invariant, and
\(\nabla P_tg=e^{-t}P_t(\nabla g)\). Differentiating
\(\E[(P_tg)\log(P_tg)]\), integrating \(\mathscr L\) by parts, and letting
\(t\to\infty\) therefore give
\[
\Ent_{\gamma_m}(g)
=\int_0^\infty\E\frac{\norm{\nabla P_tg}^2}{P_tg}\,dt
\le\int_0^\infty e^{-2t}
 \E P_t\left(\frac{\norm{\nabla g}^2}{g}\right)dt
=\frac12\E\frac{\norm{\nabla g}^2}{g}.
\]
The inequality is weighted Cauchy--Schwarz inside \(P_t\); the limit follows
by dominated convergence. Apply this to \(g=e^{\lambda f}\), first for
bounded smooth \(f\). If \(H(\lambda)=\log\E e^{\lambda f}\), then
\(\lambda H'(\lambda)-H(\lambda)\le\lambda^2b^2/2\).
Integration of \((H(\lambda)/\lambda)'\), starting at zero, gives
\(H(\lambda)\le\lambda\E f+\lambda^2b^2/2\). Use exponential Markov with
\(\lambda=t/b^2\), and apply the same result to \(-f\).

Expanding the entropy inequality at \(g=1+u(f-\E f)\), for bounded \(f\)
and small \(u\), gives
\(\Var f\le\E\norm{\nabla f}^2\). Smooth truncations and mollifications
prove the stated extensions. A Lipschitz function has at most linear growth,
so its Gaussian exponential moments justify these limits. Finally Jensen
gives \(\E\norm Z\le\sqrt m\), while the variance bound for the norm gives
\(m-(\E\norm Z)^2\le1\).
\end{proof}

\subsection{A Gaussian comparison sufficient for singular values}
We include a finite-array proof rather than importing a min--max theorem.
For independent centered Gaussian arrays \(X,Y\) and a smooth function \(F\),
\eqref{eq:gibp} applied to \(\sqrt tX+\sqrt{1-t}Y\) yields
\begin{equation}\label{eq:interpolation}
\frac d{dt}\E F(\sqrt tX+\sqrt{1-t}Y)
=\frac12\sum_{a,b}(\operatorname{Cov}(X_a,X_b)-
 \operatorname{Cov}(Y_a,Y_b))\E\partial_{ab}F.
\end{equation}
This follows by differentiating the two square-root factors and integrating
the factors \(X_a,Y_a\) by parts. A limiting argument covers degenerate
covariances.

\begin{lemma}[Gaussian extreme singular values]\label[lemma]{lem:singular}
Let \(G\) be an \(a\)-by-\(b\) standard Gaussian matrix, \(a\ge b\ge1\).
For \(t\ge0\),
\[
\PP\{s_{\max}(G)>\sqrt a+\sqrt b+t\}\le e^{-t^2/2},
\quad
\PP\{s_{\min}(G)<\sqrt{a-1}-\sqrt b-t\}\le e^{-t^2/2}.
\]
\end{lemma}
\begin{proof}
Take finite subsets \(\{u_i\}\subset S^{b-1}\) and
\(\{v_j\}\subset S^{a-1}\). With independent standard Gaussians
\(z\in\R\), \(g\in\R^a\), \(h\in\R^b\), define
\[
X_{ij}=v_j^TGu_i+z,\qquad Y_{ij}=g^Tv_j+h^Tu_i.
\]
Both variances are two and
\[
\operatorname{Cov}(X_{ij},X_{kl})-
\operatorname{Cov}(Y_{ij},Y_{kl})
=(1-u_i^Tu_k)(1-v_j^Tv_l)\ge0.
\]
This difference is zero for \(i=k\).
For \(r_i(x)=\beta^{-1}\log\sum_j e^{\beta x_{ij}}\), put
\(F(x)=-\alpha^{-1}\log\sum_i e^{-\alpha r_i(x)}\).
Its mixed second derivatives between different rows are
\(\alpha p_i p_k q_{ij}q_{kl}\ge0\), where \(p,q\) are the corresponding
soft-min and soft-max probabilities. Thus \eqref{eq:interpolation} gives
\(\E F(X)\ge\E F(Y)\). Sending \(\alpha,\beta\to\infty\) gives
\[
\E\min_i\max_jv_j^TGu_i
\ge\E\max_jg^Tv_j+\E\min_i h^Tu_i.
\]
The common additive \(z\) has mean zero and cancels. Dense finite nets of
the compact spheres give
\(\E s_{\min}(G)\ge\E\norm g-\E\norm h\).
The convergence is dominated, for example by
\(\norm G_{\Frob}+\norm g+\norm h\).

For a single soft maximum over all pairs, its off-diagonal Hessian entries
are nonpositive and its diagonal covariance differences are zero.
Equation~\eqref{eq:interpolation} instead gives
\(\E\max_{i,j}X_{ij}\le\E\max_{i,j}Y_{ij}\).
Dense nets imply \(\E s_{\max}(G)\le\E\norm g+\E\norm h\).
Lemma~\ref{lem:gauss-conc} bounds these means by
\(\sqrt{a-1}-\sqrt b\) and \(\sqrt a+\sqrt b\), respectively.
Both extreme singular values are one-Lipschitz in Frobenius norm:
use \(|\norm{Gu}-\norm{Hu}|\le\norm{G-H}_{\op}\) and take the appropriate
minimum or maximum over unit \(u\). Their one-sided concentration now proves
the result.
\end{proof}

\subsection{Chi-square, spherical, and finite-net estimates}
\begin{lemma}\label[lemma]{lem:sphere}
For \(Z\sim N(0,I_m)\) and \(0<u<1\),
\[
\PP\{|\norm Z^2/m-1|>u\}\le2e^{-mu^2/8}.
\]
If \(v\) is uniform on \(S^{m-1}\), then for every fixed \(y\),
\begin{equation}\label{eq:sphere}
\PP\{|v^Ty|\ge t\}\le2\exp\{-mt^2/(2\norm y^2)\}.
\end{equation}
For independent uniform \(v_1,\ldots,v_m\in S^{m-1}\),
\begin{equation}\label{eq:column-sphere}
\PP\{\norm{[v_1\ \cdots\ v_m]}_{\op}>4\}\le C e^{-cm}.
\end{equation}
Every sphere \(S^{m-1}\) has a \(1/4\)-net of cardinality at most \(9^m\).
For a symmetric matrix \(A\), its norm is at most twice
\(\max_{v\text{ in the net}}|v^TAv|\).
\end{lemma}
\begin{proof}
Gaussian integration gives
\(\E e^{t\norm Z^2}=(1-2t)^{-m/2}\), \(t<1/2\).
Exponential Markov, with the usual choices \(t=u/(2(1+u))\) and the
negative counterpart, and the elementary inequalities
\(u-\log(1+u)\ge u^2/4\), \(-u-\log(1-u)\ge u^2/2\), give the first claim.
Rotational symmetry gives \(Z=Rv\) with independent radius and direction.
The same generating function gives
\(\E R^{2j}=m(m+2)\cdots(m+2j-2)\), whereas
\(\E Z_1^{2j}=(2j-1)!!\). Thus
\[
\E(v^Ty)^{2j}\le(2j-1)!!\,\norm y^{2j}/m^j.
\]
Odd moments vanish. Expanding the exponential gives
\(\E e^{\lambda v^Ty}\le e^{\lambda^2\norm y^2/(2m)}\), which yields
\eqref{eq:sphere}.

For \eqref{eq:column-sphere}, write the columns as independent
\(g_i/\norm{g_i}\). Lemma~\ref{lem:singular} gives
\(\norm{[g_1\cdots g_m]}\le(5/2)\sqrt m\) except with probability
\(e^{-m/8}\). The chi-square bound and a union bound give
\(\min_i\norm{g_i}\ge(3/4)\sqrt m\) except with probability
\(C e^{-cm}\), after changing constants. Their ratio is at most \(10/3<4\).

For the net take a maximal \(1/4\)-separated set. Disjoint radius-\(1/8\)
balls about its points lie in the radius-\(9/8\) ball, giving cardinality
\(9^m\); maximality gives covering. For a unit norm-attaining eigenvector
\(x\), select a net point \(v\) with \(\norm{x-v}\le1/4\).
Then \(|x^TAx-v^TAv|\le2\norm A\norm{x-v}\le\norm A/2\).
\end{proof}

\section{Adaptive Information in a Rectangular Wishart Matrix}\label{sec:posterior}
This section treats arbitrary adaptive products, not only Krylov vectors.

\begin{lemma}[Deferred Gaussian factorization]\label[lemma]{lem:deferred}
Let \(X\in\R^{D\times(D-k)}\) have independent \(N(0,1/D)\) entries,
\(1\le k<D\), and let \(W=XX^T\). After \(q\le D-k\) adaptive orthonormal
queries \(v_1,\ldots,v_q\) to the oracle \(v\mapsto Wv\), one may choose
an orthogonal row basis, beginning with the queries, and a latent orthogonal
column basis, in which the factor is
\begin{equation}\label{eq:factor}
\begin{bmatrix}B_q&0\\ C_q&G_q\end{bmatrix}.
\end{equation}
The square \(B_q\) is invertible almost surely; \(B_q,C_q\) are determined by
the observed transcript. Conditional on an explicitly enlarged sigma-field
containing that transcript, \(G_q\) is an independent
\((D-q)\)-by-\((D-k-q)\) iid \(N(0,1/D)\) matrix. The enlargement is not
revealed to the policy.
\end{lemma}
\begin{proof}
Include the independent seed in the initial sigma-field. Orthogonalize each
new informative query against previous queries. Its reply is obtainable
from one new product and the stored previous products. A dependent query
adds nothing and can be omitted.

At step zero the assertion is the original iid factor. Suppose it holds at
step \(t\). Conditional on the enlarged field, rotate the remaining row
space so that the new query is its first basis vector. This rotation is
measurable with respect to that field, since the new query is a function of
the actual transcript. It preserves the iid conditional law of \(G_t\).
Let its first row be \(z^T\). Rotate the remaining column space so that
\(z^T=(\rho,0,\ldots,0)\), with \(\rho>0\) almost surely. Conditional on
\(z\), every other row splits into its projection on \(z/\rho\) and an
independent Gaussian row in the orthogonal complement. The new product,
after subtracting the known \(C_tC_t^T\) contribution, reveals \(\rho^2\)
and \(\rho\) times these projections. Expose \(z\) and the projections in
the enlarged field. The remaining orthogonal rows form \(G_{t+1}\), still
independent Gaussian conditional on that field. This proves the induction.
Any subsequent transcript-measurable change of basis in the remaining row
space preserves the same law.

For explicit measurability, choose every complement by a fixed
Gram--Schmidt rule. In the query row basis, \(B_t\) can be the positive
Cholesky factor of \((v_i^TWv_j)_{i,j\le t}\), and
\(C_t=(U_t^TW[v_1\cdots v_t])B_t^{-T}\), where \(U_t\) is the selected
orthogonal complement. These are actual-transcript functions. The latent
column rotations do not affect \(W\) or any observed product. Integrating
back over the enlargement proves every conditional assertion about the
actual transcript used below. A policy trying to select a new null query
before \(D-k\) independent rows have been exposed has probability zero:
its new Gaussian row is nonzero almost surely. Thus the invertibility
assertion covers adaptive policies as well.
\end{proof}

\begin{lemma}[Posterior overlap, arbitrary output rank]\label[lemma]{lem:overlap}
Under Lemma~\ref{lem:deferred}, after at most \(q\le D-k\) products, every
transcript-measurable orthogonal projector \(R\) of rank at most \(s\) obeys
\begin{equation}\label{eq:overlap}
\boxed{\quad \E\tr(R\Pi_{\ker W})\le\frac{2k(q+s)}{D-q}.\quad}
\end{equation}
Outputs and queries can be arbitrary measurable functions of the transcript
and the independent seed.
\end{lemma}
\begin{proof}
In \eqref{eq:factor}, the nullspace consists of
\[
(-B_q^{-T}C_q^Tv,v),\qquad v\in\ker G_q^T.
\]
If \(V\) is an orthonormal basis for \(\ker G_q^T\), the bottom-right block
of its full-space projector is
\[
V\bigl(I+V^TC_qB_q^{-1}B_q^{-T}C_q^TV\bigr)^{-1}V^T\preceq VV^T.
\]
Conditional on the enlarged field, \(\ker G_q^T\) is rotation-invariant of
rank \(k\) in dimension \(D-q\); hence \(\E VV^T=kI/(D-q)\).
The same bound holds after conditioning only on the actual transcript.
Before the \(t\)th orthogonal query, its expected nullspace mass is therefore
at most \(k/(D-t+1)\). The expected total mass of the final query span is
at most \(\sum_{t=1}^q k/(D-t+1)\le kq/(D-q)\).

A positive semidefinite block matrix \(K\) satisfies
\(K\preceq2\diag(K_{11},K_{22})\): this follows from
\(2|x^TK_{12}y|\le x^TK_{11}x+y^TK_{22}y\).
Apply it to \(\E[\Pi_{\ker W}\mid\mathcal T_q]\) in the query/complement
basis. Since \(0\preceq R\preceq I\), the top-block trace against \(R\) is
at most the posterior total mass of the query span. The bottom trace is at
most \(ks/(D-q)\). Take expectations and multiply by two. Fewer or dependent
queries can be padded with orthogonal ones and the output left unchanged;
extra information only strengthens the bound's applicability.
\end{proof}

\section{Sequential Readouts of a Continuous Matrix}\label{sec:matrix}

\subsection{The presampled law and its analytic event}
Some fixed scalar constants used in the objective can be chosen now. Let
\(S(t)=35t^4-84t^5+70t^6-20t^7\) on \([0,1]\), with constant extensions
zero and one. Set \(\chi(u)=S(4u-1)\), and choose
\[
H_0=2^{20},\qquad \alpha=H_0/32,\qquad B_0=16H_0,
\qquad \gamma=\alpha/(6B_0)=1/3072.
\]
For example \(\norm{\chi'}_\infty\le12\); the deliberately loose fixed
constants above make all inequalities below valid. Only their absolute
nature matters.

Choose a sufficiently large absolute integer \(\kappa\). For integers
\(J\ge4\) and \(r\ge M:=1+4096J^2\), put
\[
k=\kappa r,\quad D=kJ,\quad \eta=1/J.
\]
Let \(P\) project onto the first \(r\) coordinates of \(\R^D\), and define
\[
\mathsf D=I+(\sqrt\eta-1)P.
\]
Draw \(X\in\R^{D\times(D-k)}\) iid \(N(0,1/D)\), and set
\begin{equation}\label{eq:matrix-law}
W=XX^T,\qquad A=I+256J^2\mathsf D W\mathsf D,
\qquad \mathcal K=\ker W,\quad \mathcal Q=\mathsf D^{-1}\mathcal K.
\end{equation}
Use the event
\begin{equation}\label{eq:good}
\mathcal E=\left\{\norm W\le16,\quad
\lambda_{\min}(X^T(I-P)X)\ge\frac1{64J^2},\quad
\norm{P\Pi_{\mathcal K}P-\eta P}\le\eta/24\right\}.
\end{equation}
The law of a leaf matrix is the law in \eqref{eq:matrix-law} conditioned on
\(\mathcal E\), once before any query. This is a condition on \(W\): the
positive eigenvalues in the middle test are those of \((I-P)W(I-P)\), whose
rank is \(D-k\) almost surely.

\begin{lemma}[Positive-probability source and compressed kernel]\label[lemma]{lem:good}
For sufficiently large absolute \(\kappa,J_0\), and all \(J\ge J_0\),
\(r\ge M\), one has \(\PP(\mathcal E)\ge.99\). On this event,
\begin{equation}\label{eq:analytic-matrix}
I\preceq A\preceq MI,\qquad
A\succeq\Pi_{\mathcal Q}+5(I-\Pi_{\mathcal Q}),\qquad
\spec(P\Pi_{\mathcal Q}P|_{\range P})\subseteq[1/3,2/3].
\end{equation}
\end{lemma}
\begin{proof}
Lemma~\ref{lem:singular}, applied to \(\sqrt D X\), gives
\(\norm W\le16\) except with probability \(e^{-2D}\).
After deleting the first \(r\) rows, the standard Gaussian factor has size
\((D-r)\)-by-\((D-k)\). Since \(k\ge2(r+1)\),
\[
\sqrt{D-r-1}-\sqrt{D-k}
=\frac{k-r-1}{\sqrt{D-r-1}+\sqrt{D-k}}
\ge\frac{k}{4\sqrt D}=\frac{\sqrt D}{4J}.
\]
Taking \(t=\sqrt D/(8J)\) in that lemma proves the middle event with failure
at most \(e^{-D/(128J^2)}\).

For the last event, rotation invariance of the Gaussian factor makes
\(\mathcal K\) a rotation-invariant rank-\(k\) subspace. For every fixed
unit vector \(u\), the variable \(u^T\Pi_{\mathcal K}u\) has the same law
as
\(\sum_{j=1}^k g_j^2/\sum_{j=1}^D g_j^2\) with independent standard
Gaussians. To verify this without a separate Grassmannian theorem, average
\(u\) over an independent uniform sphere vector; conditional on any fixed
rank-\(k\) subspace the displayed ratio is the result, and rotational
invariance makes its law independent of the initial \(u\).
The chi-square bounds show, for \(\varepsilon_0=1/24\),
\[
\PP\{|u^T\Pi_{\mathcal K}u-\eta|>\varepsilon_0\eta/2\}
\le4e^{-c\varepsilon_0^2 k}.
\]
For example impose relative errors at most \(\varepsilon_0/16\) on both
chi-square sums and compare their ratio. Apply this on a \(1/4\)-net of
the unit sphere of \(\range P\). Lemma~\ref{lem:sphere} implies the required
operator bound, with failure at most \(4\cdot9^r e^{-c\varepsilon_0^2k}\).
Choose \(\kappa\) large, then \(J_0\) large. Since \(r\ge4096J^2\), the
sum of all three failures is less than \(.01\).

Write \(C=P\Pi_{\mathcal K}P\) on \(\range P\). If \(Z\) is an
orthonormal basis for \(\mathcal K\), then
\[
\Pi_{\mathcal Q}
=\mathsf D^{-1}Z(Z^T\mathsf D^{-2}Z)^{-1}Z^T\mathsf D^{-1}.
\]
Since \(\mathsf D^{-2}=I+(\eta^{-1}-1)P\), multiplication and the identity
\(R(I+R^TR)^{-1}R^T=RR^T(I+RR^T)^{-1}\) give
\begin{equation}\label{eq:compression}
P\Pi_{\mathcal Q}P=C[\eta I+(1-\eta)C]^{-1}.
\end{equation}
The scalar map \(t/(\eta+(1-\eta)t)\) has derivative at most \(1/\eta\)
for \(t\ge0\), and its value at \(t=\eta\) is \(1/(2-\eta)\).
Thus \eqref{eq:good} bounds the spectrum in \eqref{eq:compression} within
\(1/24\) of \(1/(2-\eta)\), inside \([1/3,2/3]\) when \(J\ge4\).

The nonzero eigenvalues of \(\mathsf D W\mathsf D\) equal those of
\(X^T\mathsf D^2X\succeq X^T(I-P)X\); its kernel is \(\mathcal Q\).
The spectral contribution of \(256J^2\mathsf D W\mathsf D\) on
\(\mathcal Q^\perp\) is therefore at least four. Its norm is at most
\(4096J^2=M-1\). This proves the other assertions.
\end{proof}

\subsection{Exactly computable marked readouts}
Let \(T_j\) denote the Chebyshev polynomials, defined without any oracle by
\(T_0=1\), \(T_1(t)=t\), \(T_{j+1}(t)=2tT_j(t)-T_{j-1}(t)\).
Set
\[
\delta=\min\{\gamma/32,1/1000\},\qquad
m_{\rm pol}=\left\lceil\frac{\sqrt M}{2}\log\frac2\delta\right\rceil,
\]
and define
\begin{equation}\label{eq:filter}
f_M(t)=
\frac{T_{m_{\rm pol}}((M+5-2t)/(M-5))}
{T_{m_{\rm pol}}((M+3)/(M-5))}.
\end{equation}
A mark \(w\) is uniform on the unit sphere of \(\range P\), independent
of the matrix. Its pair is
\begin{equation}\label{eq:readout}
b(w)=H_0w,\qquad
\mathsf L=(I-P)f_M(A)P,\qquad a(w)=\mathsf Lw/\norm{\mathsf Lw}.
\end{equation}
On the unconditioned source, use the same formula when the denominator is
nonzero, and otherwise use the fixed unit vector \(e_{r+1}\). This convention
makes all counterfactual algorithms finite and measurable; no test of
\(\mathcal E\) is required.

\begin{lemma}[Readout certificate and synthesis]\label[lemma]{lem:readout}
There is an absolute \(C_m\) such that \(m_{\rm pol}\le C_mJ\).
On \(\mathcal E\), for every unit mark \(w\),
\begin{equation}\label{eq:certificate}
\frac25\le\norm{\mathsf Lw}\le\frac35,\qquad
\norm{a(w)}=1,\quad Pa(w)=0,\qquad a(w)^TA^{-1}b(w)\ge H_0/4.
\end{equation}
The exact readout is computable from at most \(m_{\rm pol}\) products with
\(A\), also off \(\mathcal E\). Each such product is computable from one
product with \(W\).
\end{lemma}
\begin{proof}
For \(z_0=(M+3)/(M-5)>1\), the polynomial recurrence gives
\(T_j(\cosh u)=\cosh(ju)\); induction also gives
\(|T_j(t)|\le1\) on \([-1,1]\) using \(t=\cos\theta\).
As \(\cosh u\le1+u^2\) for \(0\le u\le1\),
\(\operatorname{arcosh}z_0\ge2/\sqrt M\). Hence the denominator in
\eqref{eq:filter} is at least
\(\frac12\exp(2m_{\rm pol}/\sqrt M)\ge1/\delta\).
We have \(f_M(1)=1\) and \(\sup_{[5,M]}|f_M|\le\delta\).
By \eqref{eq:analytic-matrix},
\begin{equation}\label{eq:filter-error}
\norm{f_M(A)-\Pi_{\mathcal Q}}\le\delta.
\end{equation}
Also \(\sqrt M\le65J\), proving the degree bound with a fixed \(C_m\).

Let \(B=P\Pi_{\mathcal Q}P\) and \(\mathsf L_0=(I-P)\Pi_{\mathcal Q}P\).
The projector identity gives \(\mathsf L_0^T\mathsf L_0=B-B^2\).
Its eigenvalues on \(\range P\) lie in \([2/9,1/4]\), so
\[
\sqrt2/3\le\nu(w):=\norm{\mathsf L_0w}\le1/2,
\qquad \norm{\mathsf L-\mathsf L_0}\le\delta.
\]
This proves the denominator bounds. Put \(a_0=\mathsf L_0w/\nu(w)\).
The inequality
\(\norm{u/\norm u-v/\norm v}\le2\norm{u-v}/\norm v\) gives
\(\norm{a(w)-a_0}\le5\delta\).
Moreover
\(a_0^T\Pi_{\mathcal Q}w=\nu(w)\),
\(A^{-1}\Pi_{\mathcal Q}=\Pi_{\mathcal Q}\), and
\(\norm{A^{-1}(I-\Pi_{\mathcal Q})}\le1/5\). Thus
\[
a(w)^TA^{-1}w\ge\sqrt2/3-1/5-5\delta>1/4.
\]
The other vector identities in \eqref{eq:certificate} follow from the formula.

Finally evaluate the Chebyshev recurrence on \(w\). Each new degree uses
one product with \(A\); projection and normalization use finite arithmetic
and a square root. On every input,
\[
Az=z+256J^2\mathsf D[W(\mathsf Dz)],
\]
so one \(W\) product suffices. The denominator-zero default has the same cap.
\end{proof}

\begin{lemma}[Unmarked tests]\label[lemma]{lem:unmarked}
For a fixed matrix in \(\mathcal E\), a fresh uniform mark \(w\), and any
\(y\) independent of \(w\) with \(\norm y\le1\),
\[
\PP\{|a(w)^Ty|>\gamma\mid A,y\}\le2e^{-c\gamma^2r}.
\]
The statement remains valid when \(y\) depends on earlier marks and their
readouts, but is selected before this mark is exposed.
\end{lemma}
\begin{proof}
A hit implies \(|w^T\mathsf L^Ty|\ge2\gamma/5\) by
\eqref{eq:certificate}, and \(\norm{\mathsf L^Ty}\le3/5\).
Apply \eqref{eq:sphere}. Conditional on the matrix and earlier data, the
new mark has exactly the original uniform law.
\end{proof}

\subsection{A chronological fork, with every earlier readout charged}
The next lemma is the probability statement actually used in the scalar
reduction. It is stronger than a lower bound for one isolated linear system.

\begin{lemma}[Fork with synthesized preprocessing]\label[lemma]{lem:fork}
Consider the conditioned matrix law above. A preprocessor may use \(q_0\)
products with \(A\), and at most \(S\) previous disclosures of
\((w_t,a(w_t))\). Previous marks may be selected from its transcript or from
independent external randomness; they contain no extra matrix-dependent
information. The preprocessor's end and disclosures may be bounded stopping
times. It may also elect to stop without opening the next mark.

After preprocessing, a fresh independent uniform mark \(w\) is supplied
without \(a(w)\). A policy uses at most \(q_1\) further products with \(A\)
and selects at most \(q_1+2\) vectors of norm at most one. It gets neither
the target readout nor any target-hit feedback while selecting these tests.
Let \(p_{\mathrm{hit}}\) be its probability that at least one test has
\(|a(w)^Ty|>\gamma\), with zero hit if it never opens that mark. Put
\[
\bar q_0=q_0+S m_{\rm pol},\qquad q=\bar q_0+rq_1.
\]
If \(q\le D-k\), then for an absolute \(C,c>0\),
\begin{equation}\label{eq:fork}
p_{\mathrm{hit}}\le
\frac{512k}{\gamma^2r\,p_{\mathcal E}}
\frac{\bar q_0+2rq_1+2r}{D-\bar q_0-rq_1}
+\frac{Ce^{-cr}}{p_{\mathcal E}},\qquad
p_{\mathcal E}=\PP(\mathcal E)\ge.99.
\end{equation}
All caps are hard caps; independent external data and an arbitrary unqueried
output test are allowed.
\end{lemma}
\begin{proof}
First impose a deterministic finite cap $N$ on internal computation,
stopping without opening a mark if the preprocessor reaches the cap and
filling unused post-mark tests with zero. Write $p_{\mathrm{hit}}^{(N)}$ for
the resulting hit probability. We prove the estimate uniformly in $N$
and remove the cap at the end.

Run the preprocessor on the unconditioned Gaussian matrix. Simulate each
ordinary product by one \(W\) product. Before using an earlier readout,
synthesize it by Lemma~\ref{lem:readout}; this costs at most
\(S m_{\rm pol}\) extra products in total. The preprocessor, including
any dependence of its stopping rule on these earlier readouts, is therefore
a genuine \(W\)-product procedure with cap \(\bar q_0\). Earlier readouts
of this same matrix have not been conditioned away or supplied as free
matrix-dependent advice.

Save this preprocessor state and run \(r\) separate post-mark branches
with independent fresh marks \(w_1,\ldots,w_r\). They do not exchange
replies, although their experiments can be simulated sequentially using the
same product oracle. Give independent copies of any post-mark private
randomness if needed. At the fixed internal-time cap, stop each branch and fill all unused
test slots with zero. Enforce the query cap on every execution. This is an
unconditioned product algorithm with cap \(q\), never testing \(\mathcal E\).

Let \(Y\) be the span of \((I-P)y\) for all tests in all branches. Its
projector is measurable from the product transcript, independent seed and
marks, and has rank at most \(r(q_1+2)\). On \(\mathcal E\) define
\[
c_i=\mathsf D\Pi_{\mathcal Q}w_i/\nu(w_i)\in\ker W,
\qquad \nu(w_i)=\norm{(I-P)\Pi_{\mathcal Q}w_i}.
\]
Since \(\mathsf D\) is the identity on \(\range(I-P)\),
\(c_i^T(I-P)y=a_0(w_i)^Ty\).
A target hit therefore implies \(\norm{\Pi_Yc_i}>3\gamma/4\), by
\(\norm{a(w_i)-a_0(w_i)}\le5\delta\le\gamma/4\).

By \eqref{eq:column-sphere}, except on an event of probability \(Ce^{-cr}\),
\(\norm{[w_1\cdots w_r]}\le4\). On its intersection with \(\mathcal E\),
\(\norm{[c_1\cdots c_r]}\le12\), since \(\norm{\mathsf D}\le1\) and
\(\nu\ge\sqrt2/3\). If \(H_{\rm hit}\) branches hit, then
\[
\frac{9\gamma^2}{16}H_{\rm hit}
\le\sum_i\norm{\Pi_Yc_i}^2
\le144\tr(\Pi_Y\Pi_{\ker W}).
\]
Thus \(\tr(\Pi_Y\Pi_{\ker W})\ge\gamma^2H_{\rm hit}/256\).
Lemma~\ref{lem:overlap} gives
\[
\E\tr(\Pi_Y\Pi_{\ker W})
\le\frac{2k(\bar q_0+2rq_1+2r)}{D-\bar q_0-rq_1}.
\]
Every capped branch has conditional hit probability $p_{\mathrm{hit}}^{(N)}$, so
$\E[\mathbf1_{\mathcal E}H_{\rm hit}]=p_{\mathcal E}r p_{\mathrm{hit}}^{(N)}$.
The exceptional spherical-column event costs at most \(Cr e^{-cr}\) in
this expectation. Combining the inequalities proves \eqref{eq:fork}
with $p_{\mathrm{hit}}^{(N)}$ on the left.
No selection uses the hidden target-hit predicate; it only defines the event
whose probability is being bounded.

On conditioned inputs, increasing the internal-time cap exhausts the
original tests. Omitted tests are filled with zero, so the hit events
increase to the original hit event, with no hit on branches that never
open the mark. Thus $p_{\mathrm{hit}}^{(N)}\uparrow p_{\mathrm{hit}}$.
Monotone convergence gives the stated bound.
\end{proof}

\begin{corollary}[A marked task has a long ordinary lifetime]\label[corollary]{cor:lifetime}
There are absolute constants \(K_0,J_0,c_\ell>0\) such that, if
\[
r\ge K_0M,\quad J\ge J_0,\quad
\ell=\lfloor c_\ell J\rfloor,
\]
and \(S\le M\), \(q_0\le M\ell+1\), \(q_1\le\ell+1\), then
\(p_{\mathrm{hit}}\le1/256\). Also \(\ell\ge(c_\ell/2)J\ge1\).
\end{corollary}
\begin{proof}
By Lemma~\ref{lem:readout},
\(\bar q_0\le M(C_m+c_\ell+2)J\).
Choose \(K_0\) large enough, then \(c_\ell\) sufficiently small, then
\(J_0\) sufficiently large. With \(D=\kappa rJ\), these choices give
\(\bar q_0+rq_1\le D/2\le D-k\). Substitution in \eqref{eq:fork} bounds
its right side by
\[
\frac{C}{\gamma^2}
\left\{\frac{C_m+c_\ell+2}{K_0}+2c_\ell+\frac4J\right\}
+Ce^{-cr}.
\]
The three choices make this at most \(1/256\). Increasing \(J_0\) makes
\(c_\ell J\ge2\). All chosen quantities depend only on fixed numerical
thresholds, not on the number of leaves or hierarchy depth.
\end{proof}

\section{Smooth Registers and the Numerical Leaf}\label{sec:leaf}

\subsection{Profiles and dimension-independent radial control}
Continue using the clamped septic \(S\) from Section~\ref{sec:matrix}, and set
\[
\chi(u)=S(4u-1),\qquad \varpi(u)=S(12(u-2/3)),\qquad
\ell_0(u)=4+\int_0^u[1-S(|t|-2)]\,dt.
\]
Every integral here and below is oriented. Thus \(\ell_0(u)=4+u\) for
\(|u|\le2\), \(3/2\le\ell_0\le13/2\), and \(0\le\ell_0'\le1\).
The profile \(\chi\) changes only in \((1/4,1/2)\); \(\varpi\) changes
only in \((2/3,3/4)\). All these functions are globally \(C^3\), and all
first three derivatives at every endpoint of a constant piece are zero.
This follows either by differentiation or from
\(S'(t)=140t^3(1-t)^3\). Their first three derivatives are bounded by
absolute constants. The displayed integrals are finite piecewise-polynomial
antiderivatives; the absolute values introduce no singularity near zero,
where the relevant integrands are constant.

For \(K>0\), define
\[
q(t)=\int_0^t[1-S(v-1)]\,dv\quad(t\ge0),\qquad
\mathcal R_K(z)=\begin{cases}
Kq(\norm z/K)z/\norm z,&z\ne0,\\0,&z=0.
\end{cases}
\]
It is the identity on \(\norm z\le K\), has norm at most \(3K/2\), and is
\(C^3\) in every dimension. Write \(\mathcal R=\mathcal R_8\).

\begin{lemma}[Radial squash bounds]\label[lemma]{lem:squash}
If \(\mathcal R_K(z)=s(r)z\), \(r=\norm z\), then \(s>0\), \(s'\le0\),
\(\norm{D\mathcal R_K}\le1\),
\(\norm{D^2\mathcal R_K}\le C/K\), and
\(\norm{D^3\mathcal R_K}\le C/K^2\), independently of dimension.
For \(r\ge K\), \(rs(r)\ge K\).
\end{lemma}
\begin{proof}
Put \(q_K(r)=Kq(r/K)\). Its derivative decreases from one to zero; hence
\(0\le q_K'\le q_K/r\le1\). The radial and tangential eigenvalues of
\(D\mathcal R_K\) are \(q_K'\) and \(s=q_K/r\), and
\(s'=(q_K'-s)/r\le0\). For \(r\ge K\),
\[
|s'|\le r^{-1},\quad
|s''|\le C/(Kr)+2/r^2,\quad
|s'''|\le C/(K^2r)+C/(Kr^2)+6/r^3.
\]
These follow by differentiating \(q_K/r\) and using bounded derivatives
of \(S\). With \(e=z/r\),
\[
D^2\mathcal R_K[v,w]
=s'\{\ip ev w+\ip ew v+\ip vw e\}
 +(rs''-s')\ip ev\ip ew e.
\]
Its norm is at most \(4|s'|+r|s''|\). Differentiation of this expression
bounds the third tensor by
\(6|s''|+6|s'|/r+r|s'''|\), proving the claims. Inside the ball the map is
linear, and flat profile endpoints give matching derivatives at both joins.
Finally \(q_K(r)\) is nondecreasing and equals \(K\) at \(r=K\).
\end{proof}

\subsection{A register block, including its scaled input}
Fix an integer \(k_0'\ge2\), a first scale \(s\ge1\), and
\(r_j=s2^{j-1}\) for \(1\le j\le k_0'\), with last scale \(R=r_{k_0'}\).
Let \(U_j\in\R^p\) be separate physical registers and \(v_j\in S^{p-1}\)
unit directions. Define
\begin{equation}\label{eq:latch-gates}
u_j=\ip{v_j}{\mathcal R(U_j/r_j)},\quad a_j=\chi(u_j),\quad
\pi=\varpi(u_{k_0'}),\quad
B(U)=8s(1-a_1)\ell_0(u_1).
\end{equation}
The self potential and its public quadratic are
\begin{align}
V(U)&=8r_1\{\tfrac12\norm{U_1/r_1}^2-a_1\ell_0(u_1)\}\notag\\
&\quad+8\sum_{j=2}^{k_0'}r_j
\{\tfrac12\norm{U_j/r_j}^2-[a_j+(1-a_j)a_{j-1}]\ell_0(u_j)\},
\label{eq:V}\\
Q_V(U)&=\sum_{j=1}^{k_0'}\frac4{r_j}\norm{U_j}^2,
\qquad W_V=V-Q_V.\notag
\end{align}
The factor \(s\) in \(B\) is part of the definition.

\begin{lemma}[Bounds and activation for a scaled register block]\label[lemma]{lem:latch}
Uniformly in \(p,k_0',s,R\) and all directions,
\begin{gather*}
0\le B\le56s,\qquad -112R\le W_V\le0,
\qquad \norm{D^jB}\le Cs^{1-j}\quad(0\le j\le3),\\
\norm{DW_V}\le C\sqrt{k_0'},\quad
\norm{D^2W_V}\le C/s,\quad \norm{D^3W_V}\le C/s^2,
\qquad \norm{D^j\pi}\le C/R^j\quad(1\le j\le3).
\end{gather*}
For every \(\lambda\ge0\),
\begin{equation}\label{eq:latch-activate}
\norm{\nabla(V-B-\lambda\pi)}\le1
\quad\Longrightarrow\quad
u_j\ge7/8\text{ for every }j,\quad\pi=1.
\end{equation}
Also, whenever \(2/3<u<3/4\) and \(\xi\le1/2\),
\begin{equation}\label{eq:transition-exclusion}
\norm{8z-(8-\xi)D\mathcal R(z)^Tv}>1,
\qquad u=\ip v{\mathcal R(z)},\quad\norm v=1.
\end{equation}
There is no lower bound on \(\xi\) in this assertion.
\end{lemma}
\begin{proof}
The negative part of a self term is at most \(56r_j\), and
\(\sum_jr_j<2R\). The local derivative factor at order \(j\) is scale to
the power \(1-j\). Each self term uses one register or two adjacent ones,
whose scale ratio is two. Split edges into odd and even pairs. Within one
collection, the variables are disjoint, so for unit directions
\(x,y,z\),
\(\sum_i\norm{x_i}\norm{y_i}\norm{z_i}\le1\); the analogous bilinear
bound also holds. This proves the second- and third-operator bounds without
multiplying by the number of registers. The gradient squares add, yielding
\(C\sqrt{k_0'}\). The other derivative bounds follow from the profiles and
Lemma~\ref{lem:squash}.

For activation, at the first register the sum of its own negative term and
\(-B\) is \(-8s\ell_0(u_1)\). Once the preceding self gate equals one,
the same cancellation occurs at each subsequent register. Its physical
gradient, apart from the last outgoing port, is
\begin{equation}\label{eq:register-gradient}
8z-[8\ell_0'(u)+d\chi'(u)]D\mathcal R(z)^Tv,\qquad d\ge0,
\end{equation}
where the next-register contribution gives
\(d=8(r_{j+1}/r_j)(1-a_{j+1})\ell_0(u_{j+1})\), or zero at the end.
The amplitude \(r_j\) cancels the derivative \(1/r_j\).
If \(\chi'(u)=0\), a gradient norm at most one implies
\(\norm z\le9/8\) from \(\norm{D\mathcal R}\le1\),
so \(\mathcal R(z)=z\), \(|u|<2\), and \(\ell_0'=1\).
The gradient is \(8(z-v)\), proving \(u\ge7/8\).
If \(\chi'(u)>0\), then \(1/4<u<1/2\). Write \(z=bv+w\),
\(w\perp v\). The radial factor is positive and nonincreasing; for the
nonnegative force coefficient in \eqref{eq:register-gradient}, the
coefficient of \(w\) is at least eight. Hence \(\norm w\le1/8\).
If \(\norm z\ge8\), then
\[
u=s(\norm z)\sqrt{\norm z^2-\norm w^2}
\ge8\sqrt{1-1/4096}>1/2,
\]
a contradiction. Inside the identity ball the longitudinal gradient is
\(8(u-1)-d\chi'(u)<-4\), another contradiction.

This induction handles all but a possibly changing last outgoing port.
If \(\varpi'=0\), it changes no gradient and the previous argument applies.
If \(\varpi'>0\), its readout lies in \((2/3,3/4)\), where its self gate
is flat one and \(\ell_0'=1\). The extra outgoing term has
\(\xi=-\lambda\varpi'/R\le0\). More generally, for any \(\xi\le1/2\)
in \eqref{eq:transition-exclusion}, the force \(8-\xi\) is positive.
The same perpendicular and radial argument puts \(z\) inside the identity
ball; its longitudinal gradient is \(8(u-1)+\xi<-3/2\).
This proves both the exclusion and the induction's last step.
\end{proof}

\subsection{A selector estimate with no dimension or task-count factor}
\begin{lemma}[Uniform stick-breaking derivatives]\label[lemma]{lem:selector}
Let \(\pi_i=\varpi(\ip{v_i}{\mathcal R(U_i/R)})\), on disjoint registers,
\(1\le i\le b\), and put
\[
\omega_i=(1-\pi_i)\prod_{j<i}\pi_j,\qquad
\omega_\infty=\prod_{j=1}^b\pi_j.
\]
The weights are nonnegative and sum to one. If the scalar payloads \(E_i\)
(and, if included, \(E_\infty\)) depend only on other coordinates and have
\(\norm{D^qE_i}\le C_0S_0^{3-q}\) for \(0\le q\le3\), then
\[
\norm{D^3\sum_i\omega_i E_i}
\le CC_0\{1+S_0/R+(S_0/R)^2+(S_0/R)^3\}.
\]
The constants do not depend on \(b\) or register dimension. Lower derivatives
have the corresponding product-rule bounds.
\end{lemma}
\begin{proof}
On smooth pieces write \(\pi_i=\sin^2\phi_i\),
\(\phi_i=\arcsin\sqrt{\pi_i}\). Near the two ends of its transition,
\(S(t)\) and \(1-S(1-t)\) are \(t^4\) times smooth strictly positive
functions. Consequently \(\phi_i\), or \(\pi/2-\phi_i\), is \(t^2\)
times a smooth function, and its first three one-sided derivatives are
bounded. After the radial composition, \(\norm{D^j\phi_i}\le C/R^j\)
on every piece, for \(j\le3\).

In a proof-only tensor space define
\(T(\phi)=\bigotimes_i(\cos\phi_i,\sin\phi_i)\).
A diagonal operator \(\mathsf E\) assigns to a bit string the payload for
its first zero, and the terminal payload to the all-one string. Then
\(\sum_i\omega_iE_i=\ip T{\mathsf E T}\).
The identity
\[
\ip{T(\phi)}{T(\psi)}=
\prod_i\cos(\phi_i-\psi_i)=\E_\varepsilon
 e^{\mathrm i\ip\varepsilon{\phi-\psi}}
\]
uses independent uniform signs \(\varepsilon_i\). Differentiating \(j\)
times in each argument shows that the squared norm of the mixed angle
derivative in directions \(a_1,\ldots,a_j\) is
\(\E\prod_{l=1}^j\ip\varepsilon{a_l}^2\).
The elementary even-moment inequality
\(\E\ip\varepsilon a^{2j}\le(2j-1)!!\norm a^{2j}\) follows by expansion:
only even powers survive and each sign moment is bounded by the corresponding
Gaussian moment. Holder gives an absolute bound times
\(\prod_l\norm{a_l}^2\), for \(j\le3\).

Composition with the disjoint scalar blocks gives \(\norm{D^jT}\le C/R^j\).
For example \(\sum_i\norm{x_i}^2\norm{y_i}^2\le\norm x^2\norm y^2\)
controls the second derivative of the angle map into \(\ell_2\); the third
is analogous. Derivative norms of the diagonal \(\mathsf E\) are the maximum
payload norms, even when their coordinates are shared. Applying the
third-order product rule to \(\ip T{\mathsf E T}\) proves the assertion.
The auxiliary angle may not be globally \(C^3\), but the actual weights are;
the uniform estimates extend from their dense smooth pieces by continuity.
The tensor space is not an objective coordinate or an oracle resource.
\end{proof}

\subsection{An inverse certificate opens the current numerical gate}
Define
\[
\vartheta(t)=\int_0^t[1-S(|v|-8)]\,dv,
\qquad U_b(y)=H_0^2\{10+\vartheta(b^Ty/H_0^2)\},
\]
\[
g_{A,b}(y)=\tfrac12y^TAy-U_b(y),\qquad
c_a(y)=\chi(a^T\mathcal R_{B_0}(y)/\alpha).
\]
Here \(\norm b=H_0\), \(\norm a=1\), and the pair satisfies
\(a^TA^{-1}b\ge H_0/4\). The primitive \(\vartheta\) equals the identity
on \([-8,8]\), has derivative between zero and one, and has absolute value
at most \(17/2\). Thus \(U_b\) is nonnegative and bounded, and its first
three derivatives are bounded by absolute constants.

\begin{lemma}[Numerical activation]\label[lemma]{lem:inverse-gate}
If \(I\preceq A\preceq MI\), \(0\le d_{\rm out}\le56\), and
\[
\norm{\nabla(g_{A,b}-d_{\rm out}c_a)(y)}\le1,
\]
then \(c_a(y)=1\), its derivative jet is flat, and
\(\norm{Ay-b}\le1\).
\end{lemma}
\begin{proof}
Denote the displayed gradient by \(g\). Since
\(\norm{\nabla U_b}\le H_0\) and
\(\norm{\nabla c_a}\le\norm{\chi'}_\infty/\alpha\),
\[
\norm y\le\norm{Ay}\le1+H_0+56\norm{\chi'}_\infty/\alpha<2H_0<B_0.
\]
Thus the radial squash is the identity and
\(|b^Ty/H_0^2|<2\), so \(\vartheta'=1\). Put \(u=a^Ty/\alpha\). Then
\(g=Ay-b-d_{\rm out}\chi'(u)a/\alpha\), and hence
\[
u=\frac{a^TA^{-1}b}{\alpha}
 +\frac{d_{\rm out}\chi'(u)}{\alpha^2}a^TA^{-1}a
 +\frac{a^TA^{-1}g}{\alpha}
\ge8-1/\alpha>1/2.
\]
The gate is on its flat-one plateau, giving the final residual assertion.
\end{proof}

\subsection{A complete leaf with a sequence of marked tasks}
Use one numerical coordinate \(y\in\R^D\). For each \(1\le s\le M\),
add \(k_\bullet\) inner registers \(X_{s,j}\in\R^p\), with independent
unit directions and scales \(r_j=2^{j-1}\). The common last scale
\(R_\bullet=2^{k_\bullet-1}\) is an absolute power of two, chosen below;
\(k_\bullet\ge2\). Write \(V_s,Q_s,B_s,\pi_s\) for the block objects
\eqref{eq:latch-gates}--\eqref{eq:V} with first scale one, and set
\[
\omega_s=(1-\pi_s)\prod_{t<s}\pi_t,
\qquad \omega_\infty=\prod_{t=1}^M\pi_t,
\qquad \mathscr B_s(y,X_s)=U_{b_s}(y)+B_s(X_s)c_{a_s}(y).
\]
Every \(\mathscr B_s\) is nonnegative, has magnitude at most an absolute
\(C_B\), and has bounded first three derivative norms. Define
\begin{equation}\label{eq:inner}
\mathcal G_A(y,X)=\tfrac12y^TAy-
 \sum_{s=1}^M\omega_s(X)\mathscr B_s(y,X_s)+\sum_{s=1}^M V_s(X_s),
\qquad o(X)=\pi_M(X_M).
\end{equation}
All marks and their readouts are fixed when this formula is formed.

\begin{lemma}[Complete inner activation and derivative estimates]\label[lemma]{lem:inner}
For a sufficiently large absolute \(R_\bullet\), and every \(\lambda\ge0\),
\begin{equation}\label{eq:inner-activate}
\norm{\nabla(\mathcal G_A-\lambda o)}\le1
\quad\Longrightarrow\quad \pi_s=1\ (1\le s\le M),\quad\norm{Ay}\le1.
\end{equation}
Writing
\(\mathcal G_A=\tfrac12y^TAy+Q_X(X)+\mathcal W_A(y,X)\),
\(Q_X=\sum_sQ_s\), one has
\begin{gather}\label{eq:inner-W}
-CM\le\mathcal W_A\le0,\quad
\norm{D\mathcal W_A}\le C\sqrt M,\quad
\norm{D^2\mathcal W_A}+\norm{D^3\mathcal W_A}\le C,\quad
\norm{\nabla_y\mathcal W_A}\le C.
\end{gather}
All constants are independent of \(M,p\).
\end{lemma}
\begin{proof}
Take the first \(s\) with \(\pi_s\ne1\). If \(\pi_s=0\), the numerical
gradient is exactly that of \(g_{A,b_s}-B_s c_{a_s}\).
Lemma~\ref{lem:inverse-gate} makes \(c_{a_s}=1\). The gradient in its block
is then that of \(V_s-B_s\), with the possible favorable outgoing term
only if \(s=M\). Lemma~\ref{lem:latch} forces \(\pi_s=1\), a contradiction.
Flat endpoint derivatives of the selectors justify this restriction even
at endpoints of their plateaus.

If \(0<\pi_s<1\), its last self gate is flat one. The dependence of the
negative forcing sum on this last register is
\(-(1-\pi_s)\mathscr B_s-\pi_s\mathscr B_{\rm tail}\), with the preceding
weights equal one. The tail is a convex combination of later forcings and
zero, and is independent of this last register. Its difference from the
current forcing has magnitude at most \(2C_B\). Thus the additional scalar
coefficient in its physical gradient has absolute value at most
\(2C_B\norm{\varpi'}_\infty/R_\bullet\), apart from a nonpositive
contribution from \(-\lambda o\) when \(s=M\).
Choose the fixed power of two \(R_\bullet\) so this is at most \(1/2\).
Equation~\eqref{eq:transition-exclusion} excludes the transition. All gates
are therefore one; the forcing weights have zero first jets and the
numerical gradient is \(Ay\).

The negative forcing is bounded by \(C_B\). Each self residual has
magnitude \(O(R_\bullet)\), proving the value bound. Apply
Lemma~\ref{lem:selector} to the \(M\) disjoint last registers, with bounded
payloads: all payload variables are separate from these last registers
because \(k_\bullet\ge2\). This bounds the first three derivatives of the
forcing sum independently of \(M\). The self residuals have disjoint blocks,
so their Hessian and third tensor norms are absolute, whereas their gradient
is \(O(\sqrt M)\). Finally the \(y\) derivative is a convex combination
of bounded \(y\) forces. Explicitly it is at most
\(H_0+56\norm{\chi'}_\infty/\alpha\) in norm. This proves
\eqref{eq:inner-W}.
\end{proof}

For all leaves use the same coordinates \(Z=(y,X)\), the same inner scales,
and the same public quadratic
\[
Q_0(Z)=M\norm y^2+Q_X(X).
\]
Their matrices, marks and directions are independent, not identified.
Let \(Y\) exceed one plus the \(y\)-force bound in
Lemma~\ref{lem:inner}, and fix an absolute \(C_*>Y+1\).
With \(\theta(y)=1-S(\norm y/(C_*M)-1)\), define the capped leaf
\begin{equation}\label{eq:capped-leaf}
E_\nu(Z)=Q_0(Z)+\tfrac12\theta(y)y^T(A_\nu-2MI)y+
\mathcal W_\nu(Z)=Q_0(Z)+D_\nu(Z).
\end{equation}
This direct use of the common inner quadratic avoids any additional
normalization of memory coordinates.

\begin{lemma}[Global leaf interface]\label[lemma]{lem:leaf}
For all leaves,
\begin{enumerate}[label=\textup{(\roman*)},leftmargin=*]
\item \(D_\nu\le0\), \(\norm{D^jD_\nu}\le CM^{3-j}\) for \(0\le j\le3\),
and \(-C_YM\le D_\nu\le0\) whenever \(\norm y\le Y\), regardless of \(X\).
\item The numerical gradient is \(B_\nu(y)y+f_\nu\), where
\(I\preceq B_\nu(y)\preceq CM I\) and \(\norm{f_\nu}\le C\).
Thus a small full gradient of any convex mixture of leaves and the public
reset, with extra terms independent of \(y\), forces \(\norm y\le Y\).
\item \(E_\nu\ge\norm y^2/2+Q_X-CM\), \(E_\nu(0)=-10H_0^2\), and
for every \(\lambda\ge0\),
\[
\norm{\nabla(E_\nu-\lambda o_\nu)}\le1\quad\Longrightarrow\quad o_\nu=1.
\]
\item Given its marked coefficients, its numerical value and gradient need
only one product \(A_\nu y\); its full numerical Hessian is computable from
\(A_\nu\). Unknown marked coefficients will be censored through exact
flat gates in the comparison argument, not supplied for free.
\end{enumerate}
\end{lemma}
\begin{proof}
Put \(C_\nu=A_\nu-2MI\preceq0\). The capped numerical gradient matrix is
\[
B_\nu(y)=2MI+\theta(r)C_\nu+
 \frac{\theta'(r)}{2r}(y^TC_\nu y)I,\qquad r=\norm y.
\]
The last scalar is nonnegative, and the first two terms are a convex
combination of \(2MI\) and \(A_\nu\), hence at least \(I\).
On the annulus \(r=\Theta(M)\), its extra norm is \(O(M)\); inside the
inner ball the formula is understood by its constant-cutoff expression.
The force \(f_\nu=\nabla_y\mathcal W_\nu\) is bounded absolutely.
All selectors and ports are independent of \(y\), proving (ii).
For \(\norm{\nabla(E_\nu-\lambda o_\nu)}\le1\), the cap is the identity
because \(\norm y\le Y<C_*M\); Lemma~\ref{lem:inner} proves (iii)'s
activation assertion. The numerical cap is pointwise a positive quadratic
mixture, at least \(\norm y^2/2\), proving its energy assertion.

The numerical correction is supported on a ball of radius \(2C_*M\).
Differentiating it \(j\) times yields \(CM(C_*M)^{2-j}\) for \(j\le3\).
At bounded \(y\) it is only \(O(M)\), with no contribution from \(Q_X\)
because that common quadratic cancels. Combine with
Lemma~\ref{lem:inner} to prove (i). Its sign follows from \(C_\nu\preceq0\)
and \(\mathcal W_\nu\le0\). Finally the quadratic form
\(y^TA_\nu y\) and every cutoff derivative use the same product \(A_\nu y\),
which proves (iv). All profiles and seams are globally \(C^3\).
\end{proof}

\section{Localized Completion Blocks and Reliability Allocation}\label{sec:localized}
The numerical matrix construction above is expensive only through its size
$M$; its long hierarchy will be hidden by a different family of completion
blocks.  The completion geometry is dimension-free analytically, while the
probability of hitting it depends on the dimension of the block.  We shall
use a larger dimension only in the first $O(\log M)$ hierarchy levels and a
fixed dimension thereafter.

\subsection{A localized completion block}
Continue to use the clamped septic step $S$ defined in
Section~\ref{sec:matrix}.  Fix $R\ge16w\ge16$, a unit vector
$v\in\R^p$, and a hidden center $c=Rv$.  For $u\in\R^p$, put
$z=u-c$, $t=\norm z$, and define
\begin{align}
 P_R(u)&=16\bigl(\sqrt{R^2+\norm u^2}-R\bigr),\label{eq:loc-public}\\
 G_c(u)&=16\sqrt{w^2+\norm{u-c}^2},\label{eq:loc-drive}\\
 K_c(u)&=P_R(c)-160w+\frac8w\norm{u-c}^2,\label{eq:loc-well}\\
 \chi_c(u)&=1-S\left(\frac{\norm{u-c}^2-4w^2}{5w^2}\right),\notag\\
 \psi_c(u)&=\chi_c(u)\bigl(K_c(u)-P_R(u)\bigr),\notag\\
 V_c(u)&=P_R(u)+\psi_c(u),\label{eq:loc-V}\\
 B_c(u)&=V_c(u)-G_c(u)+512R,\label{eq:loc-B}\\
 \pi_c(u)&=1-S\left(\frac{\norm{u-c}^2-w^2}{3w^2}\right).\label{eq:loc-port}
\end{align}
Thus $\psi_c$ is supported in $\overline B(c,3w)$; the completion port is
one on $\overline B(c,w)$ and zero outside $B(c,2w)$, with flat seam jets.
The baseline $P_R$ is public because $\norm c=R$.  The coefficient $B_c$
is deliberately not compactly supported.

\begin{lemma}[Localized values and derivative bounds]\label[lemma]{lem:localized-bounds}
All functions in \eqref{eq:loc-public}--\eqref{eq:loc-port} are globally
$C^3$.  Uniformly in the dimension and in $R,w,c$,
\begin{gather}
 -208w\le\psi_c\le0,\qquad 256R\le B_c\le768R,\label{eq:loc-values}\\
 \norm{D^j\psi_c}\le Cw^{1-j}\quad(0\le j\le3),\qquad
 \norm{D^j\pi_c}\le Cw^{-j}\quad(1\le j\le3),\label{eq:loc-derivatives}\\
 \norm{DB_c}\le C,\qquad \norm{D^2B_c}\le C/w,\qquad
 \norm{D^3B_c}\le C/w^2.\label{eq:loc-Bder}
\end{gather}
Moreover $P_R\ge0$, $P_R(0)=0$, $P_R$ is coercive, and
$\norm{D^jP_R}\le CR^{1-j}$ for $1\le j\le3$.  Finally,
\begin{equation}\label{eq:loc-cancel}
 V_c-B_c=G_c-512R.
\end{equation}
\end{lemma}
\begin{proof}
The map $P_R$ is $16$-Lipschitz.  On $t\le3w$,
\[
 K_c-P_R\le48w-160w+8t^2/w\le-40w,
 \qquad K_c-P_R\ge-48w-160w\ge-208w.
\]
Hence multiplication by $0\le\chi_c\le1$ gives the bound on $\psi_c$ and
its sign.  Regard the two square roots as Euclidean norms in one extra
coordinate.  The reverse triangle inequality gives
\[
 |P_R(u)-G_c(u)|
 \le16\{\sqrt{(R-w)^2+\norm c^2}+R\}<48R.
\]
Since $208w\le13R$, the deliberately loose estimate
$|V_c-G_c|\le256R$ holds, and \eqref{eq:loc-values} follows after adding
$512R$.

After translating by $c$ and scaling by $w$, the cutoff is a fixed smooth
function of the squared radius on a bounded annulus.  Its $j$th derivative
is $O(w^{-j})$.  On its support, the bracket $K_c-P_R$ has value $O(w)$
and $j$th derivative $O(w^{1-j})$ for $j\le3$.  The product rule proves
\eqref{eq:loc-derivatives}.  Direct differentiation of
$\sqrt{a^2+\norm z^2}$ gives first, second, and third derivative bounds
$C$, $C/a$, and $C/a^2$, respectively, in operator norm.  This proves
\eqref{eq:loc-Bder} and the estimates for $P_R$.  The septic cutoff is flat
through order three at both seams, so all estimates are global.  Identity
\eqref{eq:loc-cancel} is algebraic.
\end{proof}

\begin{lemma}[Localized activation and one-sided robustness]
\label[lemma]{lem:localized-activation}
If $\norm{\nabla(V_c-B_c)}\le1$, then
$\norm{u-c}\le w/\sqrt{255}<w$ and therefore $\pi_c=1$.  On the port
transition $w<t<2w$, for every $0\le\tau\le1$ and every $\zeta\le0$,
\begin{equation}\label{eq:loc-robust}
 \norm{\nabla V_c-\tau\nabla B_c+\zeta\nabla\pi_c}
 \ge16/\sqrt2>1.
\end{equation}
\end{lemma}
\begin{proof}
By \eqref{eq:loc-cancel},
$\norm{\nabla(V_c-B_c)}=16t/\sqrt{w^2+t^2}$.  The first assertion follows
by squaring.  On $w<t<2w$, the larger cutoff $\chi_c$ is identically one,
so
\[
 \nabla V_c=16z/w,\qquad
 \nabla G_c=16z/\sqrt{w^2+t^2},\qquad
 \nabla\pi_c=-\beta(t)z\quad(\beta(t)\ge0).
\]
Since $\nabla B_c=\nabla V_c-\nabla G_c$, the vector in
\eqref{eq:loc-robust} equals
\[
 (1-\tau)\nabla V_c+\tau\nabla G_c+\zeta\nabla\pi_c.
\]
All three terms point radially outward; the first two are a convex
combination of collinear vectors of norm at least $16/\sqrt2$.
\end{proof}

\begin{lemma}[Small support on a large sphere]\label[lemma]{lem:small-ball}
Let $v$ be uniform on $S^{p-1}$, $p\ge2$, and $R\ge16w$.  For every fixed
$u\in\R^p$,
\begin{equation}\label{eq:small-ball}
 \PP\{\norm{u-Rv}\le3w\}\le2(6w/R)^{p-1}.
\end{equation}
\end{lemma}
\begin{proof}
If the intersection is nonempty, fix one $v_0$ in it.  Any other admissible
$v$ satisfies $\norm{v-v_0}\le6w/R=:\delta\le3/8$.  Rotate $v_0$ to the
north pole.  Orthogonal projection of this spherical cap to the tangent
$(p-1)$-plane is contained in the radius-$\delta$ ball, and the inverse
Jacobian is at most two on this cap.  The cap area is therefore at most
$2\operatorname{vol}(B_{p-1})\delta^{p-1}$.  Divide by
$\operatorname{area}(S^{p-1})$ and use
$\operatorname{vol}(B_{p-1})\le\operatorname{area}(S^{p-1})$.
\end{proof}

\subsection{Reliability dimension by hierarchy level}\label{sec:allocation}
For the rest of the construction fix
\begin{equation}\label{eq:allocation}
 p_{\rm low}=\max\{8,\lceil C_p\log(2+M)\rceil\},\qquad
 j_*=2+\lceil\log_2(2+M)\rceil,
\end{equation}
and for hierarchy heights $q\ge2$ set
\begin{equation}\label{eq:level-dimension}
 p_q=\begin{cases}
 p_{\rm low},&2\le q<j_*,\\
 8,&q\ge j_*.
 \end{cases}
\end{equation}
The numerical leaf and the first outer height use $p_{\rm low}$-dimensional
registers.  Different heights always use different physical blocks.  The
analytic estimates above are independent of dimension.

At height $q$ the radius ratio below will be
$R_q/w_q=K2^{q-2}$.  Thus a fixed-point premature hit has probability
\[
 \theta_q(p_q)=2\left(\frac6{K2^{q-2}}\right)^{p_q-1}.
\]
When $p_q=8$,
\begin{equation}\label{eq:fixed-tail}
 \sum_{q=j_*}^{\infty}8^q\theta_q(8)
 =\frac{2048}{15}(6/K)^7 16^{-(j_*-2)}
 \le\frac{2048}{15}(6/K)^7(2+M)^{-4}.
\end{equation}
This identity will remove the factor $\log M$ from the physical depth cost.

\section{The First Hierarchy Level}\label{sec:first-height}
Instantiate the continuous matrix construction of Sections
\ref{sec:matrix}--\ref{sec:leaf} with
\begin{equation}\label{eq:matrix-rank-choice}
 r=r_A:=\lceil K_AM\rceil
\end{equation}
for a sufficiently large absolute $K_A$, and use register dimension
$p_{\rm low}$ in every leaf.  Draw all leaf matrices independently from
the conditioned law of Section~\ref{sec:matrix}; each leaf has its own
marks and inner directions, but all leaves share the same physical leaf
coordinate $Z=(y,X)$ and the same public quadratic $Q_0$.

Choose a power of two $R_0\in [M,2M)$ and put $R_1=4R_0$.  At a height-one
node introduce eight scaled register blocks, each with first scale one,
last scale $R_1$, and dimension $p_{\rm low}$ in every register.  Denote
their objects from Section~\ref{sec:leaf} by
$V_{v,i}^{(1)},Q_{v,i}^{(1)},B_{v,i}^{(1)},\pi_{v,i}^{(1)}$; the public
quadratics do not depend on the node $v$.  Set
\[
 Q_1=Q_0+\sum_{i=1}^8Q_i^{(1)},\qquad
 A_0=C_AM,\qquad S_1=A_0+56,\qquad
 A_1=8(S_1+112R_1),
\]
where $C_A$ is a sufficiently large absolute constant.  For a height-one
node with leaf children $vi$, define
\[
 \omega_{v,i}^{(1)}=(1-\pi_{v,i}^{(1)})
                 \prod_{j<i}\pi_{v,j}^{(1)},\qquad
 \omega_{v,\infty}^{(1)}=\prod_{j=1}^8\pi_{v,j}^{(1)},
\]
and
\begin{equation}\label{eq:height-one}
\begin{aligned}
 F_v^{(1)}={}&Q_0+
 \sum_{i=1}^8\omega_{v,i}^{(1)}
 \bigl[E_{vi}-Q_0-(i-1)S_1-B_{v,i}^{(1)}o_{vi}\bigr]\\
 &-8S_1\omega_{v,\infty}^{(1)}+\sum_{i=1}^8V_{v,i}^{(1)},
 \qquad o_v^{(1)}=\pi_{v,8}^{(1)}.
\end{aligned}
\end{equation}

\begin{lemma}[Complete first-height interface]\label[lemma]{lem:first-height}
Uniformly in the node, $M$, and $p_{\rm low}$, the following hold.
\begin{enumerate}[label=\textup{(\roman*)},leftmargin=*]
\item With $D_v^{(1)}=F_v^{(1)}-Q_1$, one has $D_v^{(1)}\le0$ and
$-A_1\le D_v^{(1)}\le0$ whenever $\norm y\le Y$.  Globally
$F_v^{(1)}\ge\norm y^2/2+Q_X+\sum_iQ_i^{(1)}-A_1-CM$.
\item For every $\lambda\ge0$,
$\norm{\nabla(F_v^{(1)}-\lambda o_v^{(1)})}\le1$ forces all eight
$\pi_{v,i}^{(1)}=1$.
\item For $j=1,2,3$,
\[
 \norm{D^jD_v^{(1)}}\le C M^{3-j}.
\]
In particular the third-derivative bound is absolute.
\end{enumerate}
\end{lemma}
\begin{proof}
On $\norm y\le Y$, Lemma~\ref{lem:leaf} gives
$E_{vi}-Q_0\in[-A_0,0]$, after increasing $C_A$, and
$0\le B_{v,i}^{(1)}o_{vi}\le56$.  Hence the $i$th relative payload lies in
$[-iS_1,-(i-1)S_1]$, while the terminal payload is $-8S_1$.
The weighted payloads are a convex combination.  Each self residual is in
$[-112R_1,0]$, proving the first bounded-$y$ assertion.  The global lower
bound keeps the common positive quadratics and uses the leaf lower bound;
only one leaf perturbation is selected by the convex weights.

The numerical $y$ gradient is a convex mixture of the positive matrices and
bounded forces of Lemma~\ref{lem:leaf}; all new objects are independent of
$y$.  Thus a full gradient of norm at most one first implies $\norm y\le Y$.
Take the first $i$ with $\pi_{v,i}^{(1)}\ne1$.  If it equals zero, the lower
coordinate gradient is that of
$E_{vi}-B_{v,i}^{(1)}o_{vi}$.  Lemma~\ref{lem:leaf} forces $o_{vi}=1$,
and the incoming-one latch implication \eqref{eq:latch-activate} contradicts
$\pi_{v,i}^{(1)}=0$.  If $0<\pi_{v,i}^{(1)}<1$, the current and tail
payloads satisfy $H_{\rm tail}\le H_i$.  Since the first and last latch
registers are distinct, the only new term in the last-register force has
coefficient
\[
 \xi=\frac{\varpi'(u)}{R_1}(H_{\rm tail}-H_i)
 -\mathbf1_{\{i=8\}}\frac{\lambda\varpi'(u)}{R_1}\le0.
\]
The transition exclusion \eqref{eq:transition-exclusion} gives the
contradiction.

For derivatives, the leaf relative payloads have scales
$CM^{3-j}$ by Lemma~\ref{lem:leaf}; selector derivatives are at scale
$R_1\asymp M$.  Lemma~\ref{lem:selector} bounds the weighted leaf part.
The eight incoming products have bounded derivatives through order three
because $B_{v,i}^{(1)}$ uses the first register, $o_{vi}$ uses leaf
coordinates, and there are only eight terms.  Lemma~\ref{lem:latch} bounds
the self residuals.  The product rule gives the stated three scales.
\end{proof}

\section{A Hierarchy with Heterogeneous Register Dimensions}\label{sec:hierarchy-log}
Let $n=8^L$.  Height zero consists of the leaves; height one is
\eqref{eq:height-one}.  At every height $q\ge2$ choose one sufficiently
large absolute geometric constant $K\ge384$ and set
\begin{equation}\label{eq:high-scales}
 w_q=R_0 3^{q-2},\qquad R_q=K R_0 6^{q-2}.
\end{equation}
At that height introduce eight physical vectors
$U_i^{(q)}\in\R^{p_q}$.  All nodes of a fixed height share these physical
blocks, but each node $v$ and child index $i$ has its own independent
uniform center
\[
 c_{v,i}^{(q)}=R_qv_{v,i}^{(q)},\qquad
 v_{v,i}^{(q)}\sim\operatorname{Unif}(S^{p_q-1}).
\]
Use the localized objects
$V_{v,i}^{(q)},B_{v,i}^{(q)},\pi_{v,i}^{(q)}$ with width $w_q$.
Define the common baseline recursively by
\begin{equation}\label{eq:common-baseline}
 Q_q=Q_{q-1}+\sum_{i=1}^8P_{R_q}(U_i^{(q)}).
\end{equation}
Set $C_B=768$, $C_W=208$, and
\begin{equation}\label{eq:energy-rec}
 S_q=A_{q-1}+C_BR_q,
 \qquad A_q=8S_q+8C_Ww_q\qquad(q\ge2).
\end{equation}
For a node $v$ at height $q$, put
\[
 \omega_{v,i}=(1-\pi_{v,i})\prod_{j<i}\pi_{v,j},\qquad
 \omega_{v,\infty}=\prod_{j=1}^8\pi_{v,j},
\]
and define the literal full-space scalar
\begin{equation}\label{eq:hierarchy-recursion}
\boxed{\begin{aligned}
 F_v^{(q)}={}&Q_{q-1}
 +\sum_{i=1}^8\omega_{v,i}
 \bigl[F_{vi}^{(q-1)}-Q_{q-1}-(i-1)S_q
       -B_{v,i}o_{vi}^{(q-1)}\bigr]\\
 &-8S_q\omega_{v,\infty}+\sum_{i=1}^8V_{v,i},
 \qquad o_v^{(q)}=\pi_{v,8}.
\end{aligned}}\end{equation}
The root is $F=F_\varnothing^{(L)}$.  All matrices, marks, inner directions,
and localized centers are sampled before the algorithm's private seed.

\begin{lemma}[Ordered values, activation, energy, and coercivity]
\label[lemma]{lem:hierarchy-analytic}
Let $D_v^{(q)}=F_v^{(q)}-Q_q$.  Uniformly over all finite depths and all
inputs:
\begin{enumerate}[label=\textup{(\roman*)},leftmargin=*]
\item $D_v^{(q)}\le0$, and $-A_q\le D_v^{(q)}\le0$ whenever
$\norm y\le Y$.  Globally $D_v^{(q)}\ge-A_q-CM^3$.
\item $F_v^{(q)}$ is coercive and is bounded below by $\norm y^2/2$ plus
all its common positive memory baselines, minus $A_q+CM$.
\item For every $\lambda\ge0$,
$\norm{\nabla(F_v^{(q)}-\lambda o_v^{(q)})}\le1$ forces all eight
immediate completion ports to one.
\item $A_q\le C8^qM$, and for the root
$F(0)-\inf F\le CnM$.
\end{enumerate}
\end{lemma}
\begin{proof}
The numerical $y$ derivative remains a convex mixture of the leaf positive
matrices and bounded leaf forces.  Every object added above the leaves is
independent of $y$.  A small full gradient therefore first gives
$\norm y\le Y$.

On that region the $i$th relative payload
\[
 H_i=D_{vi}^{(q-1)}-(i-1)S_q-B_{v,i}o_{vi}^{(q-1)}
\]
lies in $[-iS_q,-(i-1)S_q]$, while the terminal payload is $-8S_q$.
The weights including the terminal weight form a convex partition of one.
Every localized self residual $V_{v,i}-P_{R_q}$ belongs to
$[-C_Ww_q,0]$.  This proves the bounded-$y$ range and, in particular, the
ordered-tail inequality $H_{\rm tail}\le H_i$.  The coarse global
$-CM^3$ term is carried through a convex combination only once.  For the
true global lower bound, retain the positive capped numerical term and all
common memory baselines rather than using that coarse relative estimate.
Every finite baseline is coercive.

For activation, take the first $i$ with $\pi_{v,i}\ne1$.  If
$\pi_{v,i}=0$, the lower-coordinate gradient is exactly that of
$F_{vi}^{(q-1)}-B_{v,i}o_{vi}^{(q-1)}$; here $B_{v,i}\ge0$ is constant in
those lower coordinates.  The induction hypothesis forces the child port
to one.  The $U_i^{(q)}$ gradient is then $\nabla(V_{v,i}-B_{v,i})$, plus
a favorable outgoing term only when $i=8$, contradicting
Lemma~\ref{lem:localized-activation}.  If $0<\pi_{v,i}<1$, differentiation
in the same parent block gives exactly
\[
 \nabla V_{v,i}-(1-\pi_{v,i})o_{vi}\nabla B_{v,i}
 +(H_{\rm tail}-H_i-\mathbf1_{\{i=8\}}\lambda)\nabla\pi_{v,i}.
\]
The first coefficient lies in $[0,1]$ and the final scalar coefficient is
nonpositive.  Equation~\eqref{eq:loc-robust} excludes a small gradient.
The height-one base is Lemma~\ref{lem:first-height}.

Finally,
\[
 \frac{A_L}{8^L}=\frac{A_1}{8}
 +\sum_{q=2}^L\frac{C_BR_q+C_Ww_q}{8^{q-1}}\le CM,
\]
because $6<8$ and $3<8$.  At zero, all completion gates and localized
self residuals vanish and the recursion selects the first leaf with zero
offset.  Its value is $O(M)$, while the global lower bound is $-A_L-O(M)$.
This proves the gap.
\end{proof}

\section{Uniform Third-Derivative Bounds}\label{sec:regularity-log}
At height one, Lemma~\ref{lem:first-height} gives relative derivative scales
$CM^2,CM,C$.  For $q\ge2$, let $a_{j,q}$ bound
$\norm{D^jD_v^{(q)}}$ over all nodes and points, and put
$J_q=C(M^3+M8^q)$.  Let $b_{q-1}$ be the scale of the child output port:
$b_1=R_1$ and $b_{q-1}=w_{q-1}$ for $q\ge3$.  Then $b_{q-1}$ and $w_q$
differ by an absolute factor.

\begin{lemma}[Incoming-coupling derivative cost]\label[lemma]{lem:coupling-cost}
For a term $\omega_iB_io_{vi}$ at height $q\ge2$, all product-rule terms
that are new at that height have total bounds
\[
 C(1+R_q/b_{q-1}),
\]
\[
 C(w_q^{-1}+b_{q-1}^{-1}+R_q/b_{q-1}^2),
\]
and
\[
 C(w_q^{-2}+b_{q-1}^{-2}+R_q/b_{q-1}^3)
\]
at derivative orders one, two, and three.  The localized self residuals
have order-$j$ bound $Cw_q^{1-j}$.  The constants do not depend on the
register dimension or descendant count.
\end{lemma}
\begin{proof}
There are eight terms.  The selector derivatives have size
$Cw_q^{-j}$, Lemma~\ref{lem:localized-bounds} gives
$|B_i|\le CR_q$ and its positive-order bounds, and the child port has
$\norm{D^jo_{vi}}\le Cb_{q-1}^{-j}$.  Apply the product rule.  Dependence
of $B_i$ and $\omega_i$ on the same parent vector has been included; parent
and child variables are distinct.  The self-residual estimates are
\eqref{eq:loc-derivatives}.
\end{proof}

\begin{theorem}[Absolute Hessian regularity]\label{thm:regularity-log}
Uniformly in depth, input, and the heterogeneous register dimensions,
\begin{equation}\label{eq:relative-derivatives}
 a_{1,q}\le C\{M^2+(8/3)^q\},\qquad
 a_{2,q}\le CM,\qquad a_{3,q}\le C.
\end{equation}
Consequently the root is globally $C^3$ and
$\Lip(\nabla^2F)\le C_H$ for one absolute constant $C_H$.
\end{theorem}
\begin{proof}
The terms in which no derivative lands on a selector are a convex
combination of child derivative tensors, so the inherited highest-order
coefficient is one.  All other inherited terms contain at least one
selector derivative.  Lemma~\ref{lem:coupling-cost} and the public offsets
give
\begin{align*}
 a_{1,q}&\le a_{1,q-1}+C(J_q/w_q+1+R_q/b_{q-1}),\\
 a_{2,q}&\le a_{2,q-1}+C(a_{1,q-1}/w_q+J_q/w_q^2
 +w_q^{-1}+b_{q-1}^{-1}+R_q/b_{q-1}^2),\\
 a_{3,q}&\le a_{3,q-1}+C(a_{2,q-1}/w_q+a_{1,q-1}/w_q^2+J_q/w_q^3
 +w_q^{-2}+b_{q-1}^{-2}+R_q/b_{q-1}^3).
\end{align*}
Since $w_q\asymp M3^q$ and $R_q\asymp M6^q$, the first increments are
bounded by
$C\{M^2 3^{-q}+(8/3)^q+2^q+1\}$, proving the first estimate.  Substitution
into the second recurrence produces the summable terms
\[
 C\{M3^{-q}+M9^{-q}+M^{-1}(8/9)^q+M^{-1}3^{-q}+M^{-1}(2/3)^q\},
\]
and substitution into the third produces
\[
 C\{3^{-q}+9^{-q}+27^{-q}+M^{-2}(8/27)^q
      +M^{-2}(2/9)^q+M^{-2}9^{-q}\}.
\]
This proves \eqref{eq:relative-derivatives}; the exceptional height two is
absorbed into the absolute constants.  The common baseline at height $q$
contains eight copies of $P_{R_q}$ on disjoint blocks, each with third norm
$C/R_q^2$.  The trilinear norm of a sum on disjoint blocks is bounded by the
maximum block norm, by H\"older's inequality.  Each finite formula is $C^3$;
integration of the third derivative along line segments gives the global
Hessian-Lipschitz bound.
\end{proof}

\section{Exact Oracle Replies with Charged Repairs}\label{sec:repairs}
We now define a stronger oracle which always returns the literal jet of the
fixed presampled scalar.  It may disclose input data prematurely; every
marked task thereby retired is charged to a random variable $Z$.  This
replaces the global no-error flatness event used in simpler constructions.

\subsection{Presampling, retirement states, and repair loop}
Presample independently all $n$ leaf matrices from their conditioned laws,
their $M$ marks and inner directions, all height-one directions, and all
localized centers at higher nodes.  A leaf is retired or unretired.  The
current leaf is the first unretired leaf in lexicographic order, and its
marks retire in order.  A node is retired when all its leaves are retired.
Every leaf retirement, whether by normal completion, a rich query, or any
repair, discloses its full matrix, all marks and readouts, and every inner
direction before the reply is assembled. This rule applies to every leaf
in a subtree retired by a repair. Disclose all completion data within a
retired subtree as well, so that its entire scalar and output port are known.
When a child retires, disclose the parent completion data associated with
that child; propagate disclosures through newly completed ancestors.  A
normal rich query retires only its current leaf. Initially, before the
first point is selected, supply the first mark of the first leaf without
its readout. Whenever an unretired leaf first becomes current, supply its
first mark on the same convention. Supply of a mark starts its ordinary
slot counter at zero.

For an old or inner register with physical coordinate $U$, scale $r$, and
hidden unit direction $v$, the repair test is exactly
\begin{equation}\label{eq:memory-repair-test}
 \left|\left\langle v,\mathcal R(U/r)\right\rangle\right|>\frac18.
\end{equation}
Here $\mathcal R=\mathcal R_8$ is the radial squash already defined in
\Cref{sec:leaf}; there is no further division by a vector norm.

A node is visible at a selected point when it lies on the root-to-current
leaf chain in the current retirement state.  Before assembling any ordinary
or rich reply, fix the physical point and repeatedly apply the following
publicly ordered repair rules until none applies.
\begin{enumerate}[label=\textup{(\arabic*)},leftmargin=*]
\item At a visible node of height $q\ge2$, if an undisclosed center satisfies
$\norm{U_i^{(q)}-c_{v,i}^{(q)}}\le3w_q$, retire and disclose its whole child
subtree.  If $i=8$, retire the whole node.  Charge every newly retired marked
task to $Z$.
\item At height one, if an undisclosed old direction satisfies
\eqref{eq:memory-repair-test} in its register, retire the whole height-one
node and charge its newly retired tasks to $Z$.
\item In the current leaf, if an undisclosed inner direction satisfies
\eqref{eq:memory-repair-test}, or a not-yet-supplied mark $w$ satisfies
$|a(w)^TY(y)|>\gamma$, retire the leaf and charge its remaining tasks to $Z$.
\end{enumerate}
Whenever a repair opens a new current leaf at the same point, supply its
first mark and repeat the loop.  A fixed priority--ancestors first, then child,
register, and mark indices--makes the procedure deterministic given the raw
input.  Every repair retires at least one leaf, so the loop is finite.
After the loop, if a leaf remains, perform exactly one normal operation: an
ordinary query uses one current matrix product. The current marked task
retires at its first test $|a(w)^TY(y)|>\gamma$, or after $\ell$ ordinary
slots since its supply, whichever occurs first. At retirement, disclose
its readout and all directions in its corresponding inner register block,
including registers that were not hit. If it is the last mark, retire the
leaf under the full-disclosure rule above; otherwise supply the next mark
in this same reply. A rich query discloses and retires the entire current
leaf. Newly completed ancestor data and the first mark of any new current
leaf are supplied in the same reply. Marks supplied by a normal operation
are available for subsequent point selections; no second normal operation
is performed at the already selected point. A last-child latch is never disclosed
while its node remains unretired: an out-of-order hit of that latch retires
the whole node.

For numerical tests use the bounded vector
\begin{equation}\label{eq:bounded-test}
 Y(y)=\frac{\mathcal R_{B_0}(y)}{3B_0/2},\qquad \norm{Y(y)}\le1.
\end{equation}
The current mark's hidden readout is used internally only to decide its
normal hit event and is disclosed when that task retires.

\begin{lemma}[Every reply is the true jet]\label[lemma]{lem:true-jet}
The repaired comparison oracle returns the literal ordinary or rich jet of
the fixed presampled root scalar at every selected point.  Apart from
matrices disclosed through charged repairs, an ordinary reply needs at most
one product with the current matrix and a normal rich reply needs only that
one new matrix.  No unretired sibling matrix is contained for free in a
normal rich reply.
\end{lemma}
\begin{proof}
After the repair loop, every still-hidden localized self perturbation at a
visible node vanishes on a neighborhood of the point, and its port has zero
complete second jet.  At the first unfinished child, that hidden port
annihilates every later sibling weight.  Earlier and prematurely retired
children are fully known.  Recursing reaches one current leaf.  The same
readout tests make every still-hidden height-one and inner self/port jet
zero; future numerical gates are flat.  The only unknown ordinary numerical
term is the current matrix product.

The noncompact incoming coefficient $B_{v,i}$ requires a separate check.
While it is hidden, its child output port has zero complete second jet; if
that port becomes nonflat, a descendant repair applies before the reply.
When the child retires, the center defining $B_{v,i}$ is disclosed in the
same reply.  Hence the true second jet of $B_{v,i}o_{vi}$ is always either
zero without knowing $B_{v,i}$ or exactly computable from disclosed data.

For the currently supplied numerical mark, the comparison oracle first tests its
hidden hit predicate.  If the test fires, or if the ordinary timeout is reached,
its readout is disclosed before the reply is assembled; otherwise the gate is on
its flat-zero plateau at the selected point.  A rich operation discloses the
whole current leaf before evaluating its full jet.

At a normal completion, every newly disclosed block was tested before the
normal operation at the already selected point.  Its port remains flat zero
at that point, so the next leaf does not enter the same reply.  If repairs
opened additional branches, they were completely processed before the one
normal operation and their retired matrices were charged to $Z$.
\end{proof}

\subsection{Local visibility windows and stopped tests}
\begin{lemma}[Visibility-window cap]\label[lemma]{lem:window}
A node containing $8^q$ leaves is visible at at most
$C8^qM\ell$ selected points before retirement, including a virtual final
output.  A leaf is visible at at most $CM\ell$ points.  The same deterministic
caps hold if one specified repair test is suppressed.
\end{lemma}
\begin{proof}
The frontier is the first unretired leaf and never moves backward.  During a
node's visible interval, each normal ordinary slot belongs to one of its
$8^qM$ marked tasks, each with timeout $\ell$; each normal rich query retires
one of its leaves.  Repairs only shorten the interval.  The finitely many
boundary points created within a reply are absorbed into the same bound.
Suppressing one repair does not affect the normal timeouts.
\end{proof}

\begin{lemma}[Stopped fresh-parameter tests]\label[lemma]{lem:stopped-local}
Fix one hidden localized center, one old or inner memory direction, or one future
numerical mark; in the last case also condition on its matrix.  Its probability
of causing a repair is at most its deterministic visibility-window cap times
the corresponding fixed-point hit probability.  This estimate is
unconditional in previous misses.
\end{lemma}
\begin{proof}
Condition on all other raw variables and the policy seed. Run a shadow
comparison process with the selected parameter's repair test suppressed.
In addition, eliminate that parameter from every reply by the following
substitutions, made as functions before taking any derivative.
\begin{enumerate}[label=\textup{(\roman*)},leftmargin=*]
\item For a localized center, replace its objects by
$(V,\pi,B)=(P_R,0,0)$, including every occurrence of its incoming
coefficient $B$.
\item For an old or inner memory direction $v$, replace its readout
$U\mapsto\langle v,\mathcal R(U/r)\rangle$ by the identically zero
function everywhere in its block formulas, while retaining the public
quadratic. Its derivative contributions are therefore zero as well.
\item For a future mark $w_s$, replace $U_{b_s}$ and $c_{a_s}$ by zero
until the mark is supplied; equivalently, omit its full weighted forcing
term $\omega_s\mathscr B_s$ from \eqref{eq:inner}.
\end{enumerate}
The shadow uses these substituted jets and the resulting history for all
later choices. It does not observe the suppressed test's outcomes.
These are counterfactual replies; they need not arise from a legal scalar.

Before visibility, the selected parameter has no effect on a reply.
Before its first hit during visibility, a localized self perturbation and
port have zero second jet; the possibly nonzero incoming coefficient
multiplies a child port with zero second jet. An undisclosed old or inner
direction belongs to a block whose directions are all still hidden. At a
reply preceding any repair or disclosure of that block, their self and
port gates are flat zero; its incoming coefficient multiplies a flat-zero
child or numerical gate. For a future mark, a preceding unfinished task
has a flat-zero inner port, so the mark's weighted forcing has zero second
jet even though its unweighted forcing need not vanish.
Thus the real and censored replies agree until the first selected-parameter
hit or prescribed disclosure. Other repairs use only the fixed other
variables and the censored history. Normal disclosure is triggered by
completion of the associated child or task, which uses no undisclosed
completion parameter. In the future-mark case, supply is determined by
the preceding tasks.

Record the selected points at which the suppressed repair test would be
performed, without feeding any outcome back to the shadow. Stop this list
whenever a normal operation or another repair would disclose the selected
parameter, before returning that disclosure or using its raw value; then
fill unused slots with zero. This includes full-data disclosures at leaf
or subtree retirement, and, for a future mark, supply of the mark itself.
The stopped list and its length are functions only of the other raw
variables and the policy seed, and obey \Cref{lem:window}. Evaluate the
selected parameter's hit predicates on this independent list only for the
analysis. Agreement before the first hit implies that every real repair
caused by that parameter appears in the list. A union bound proves the
claim without conditioning on previous misses.
\end{proof}

\section{Repair Cost and Numerical Task Lifetime}\label{sec:repair-cost}
Let $Z$ count actual marked tasks retired by repairs, each task once.
Potential overlapping repairs may be overcounted when taking expectations.

\begin{theorem}[Repair reliability independent of the number of leaves]
\label{thm:repair-cost}
Choose the absolute geometric $K$ sufficiently large, then choose $C_p$ in
\eqref{eq:allocation} and $K_A$ in \eqref{eq:matrix-rank-choice}
sufficiently large.  For every $n=8^L$ and every capped adaptive comparison
policy,
\begin{equation}\label{eq:repair-expectation}
 \E Z\le nM/256.
\end{equation}
The choices do not depend on $n$ or $L$.
\end{theorem}
\begin{proof}
At height $q\ge2$, Lemmas~\ref{lem:small-ball}, \ref{lem:window}, and
\ref{lem:stopped-local} give
\begin{equation}\label{eq:outer-repair-charge}
 \frac{\E Z_{\ge2}}{nM}
 \le C M\ell\sum_{q=2}^L8^q
     2\left(\frac6{K2^{q-2}}\right)^{p_q-1}.
\end{equation}
Indeed, there are $n/8^q$ nodes and eight centers per node; each center has
at most $C8^qM\ell$ tests and one hit retires at most $8^qM$ tasks.

For $2\le q<j_*$, extend to the infinite series with
$p_q=p_{\rm low}\ge8$:
\[
 \sum_{q=2}^{j_*-1}8^q\theta_q(p_{\rm low})
 \le256(6/K)^{p_{\rm low}-1}.
\]
Choose $C_p$ large enough that its contribution to
\eqref{eq:outer-repair-charge} is at most $1/2048$, uniformly in $M$.
For $q\ge j_*$, use \eqref{eq:fixed-tail}; since
$M\ell=O(M^{3/2})$, a sufficiently large fixed $K$ makes this contribution
at most $1/2048$ for every admissible $M$.

At height one, each node has $8k_1$ old directions, a window $O(M\ell)$,
and repair charge at most $8M$.  The normalized expected cost is at most
$Ck_1M\ell e^{-cp_{\rm low}}$.  Inner leaf directions contribute at most
$CM^2\ell e^{-cp_{\rm low}}$.  The $M$ unprovided marks in each leaf
contribute at most $CM^2\ell e^{-cr_A}$.  Lemma~\ref{lem:sphere} and
Lemma~\ref{lem:unmarked} justify these fixed-point tails.  Increasing
$C_p$ and $K_A$ makes each term at most $1/2048$.  The five contributions
sum to less than $1/256$.
\end{proof}

\begin{lemma}[The ordinary marked lifetime is unchanged]\label[lemma]{lem:normal-life}
Let $E_{\rm early}$ count normal ordinary marked-task retirements strictly
before their timeout which were not assigned to repairs.  Then
\begin{equation}\label{eq:early-expectation}
 \E E_{\rm early}\le nM/256.
\end{equation}
\end{lemma}
\begin{proof}
Fix a target leaf and target mark.  A leaf repaired before becoming current
has no normal retirements.  For an unrepaired current leaf, no product with
its matrix occurred before its visible interval.  Other matrices, localized
centers, and memory directions are independent auxiliary seed data.  Earlier
readouts of the same target matrix are not free: synthesize them with the
charged $O(J)$ products of Lemma~\ref{lem:readout} before using their hit
predicates.

Before the target mark is supplied, run a shadow policy that suppresses
every premature-repair test of a still-unprovided mark from the target
matrix. In its replies, replace both $U_{b_s}$ and $c_{a_s}$ for every
such mark by identically zero functions, including their derivatives.
Thus their full weighted forcing terms are omitted, as in
\Cref{lem:stopped-local}. Previously supplied marks retain their true
forcing and gates: their readouts can be synthesized privately with the
charged products before reproducing their normal hit tests and disclosures.
The shadow does not read the target mark before its supply. Its supply time
is determined by the earlier tasks, so it is fresh at that time. This
preprocessing uses at most $M\ell+1$ products and at most $M$ synthesized
earlier readouts.

After the target mark $w$ is supplied, retain its known forcing $U_{b(w)}$
and the normal matrix products, but replace its gate $c_{a(w)}$ by the
identically zero function. Keep the target task active until its fixed
timeout; do not execute its normal hit-retirement branch or return any
target-hit or target-miss feedback. Continue to censor all unprovided
marks. At each selected ordinary point, record $Y(y)$ before assembling
the substituted reply. A point fixed before supply, including supply
during a repair loop, is recorded without reselection. Stop before a rich
reply or any reproduced repair that would disclose the target matrix or
readout. At the timeout, record the selected point but stop before its
disclosure. Upon termination, also record any final unqueried output as
the allowed output test. Fill unused slots with zero. This gives at most
$q_1=\ell+1$ products and at most $q_1+2$ bounded tests, without the target
readout or any target-hit feedback.

Repairs depending on independent data can be reproduced from the auxiliary
seed. Until the first suppressed future-mark repair or normal target hit,
the omitted terms have zero jet at every reply, so the real and shadow
histories agree. If a suppressed future-mark repair happens first, the
real leaf retires and contributes no later normal target retirement.
Otherwise every real normal early target retirement is recorded as a hit
in the shadow's test list before the corresponding readout is disclosed.
The first-hit event of this no-feedback test list therefore contains every
real normal early target retirement. Corollary~\ref{cor:lifetime} bounds
its probability by $1/256$. Summing the $nM$ indicators proves the claim;
independence between these indicators is unnecessary.
\end{proof}

\section{The Joint Query Budget}\label{sec:ledger-log}
An algorithm promised to terminate only on legal scalar inputs can first be
capped on counterfactual comparison histories by an input-independent finite
computation cutoff.  Monotone convergence under the legal presampled input
law makes the lost average success arbitrarily small; all bounds below are
independent of this cutoff.

\begin{theorem}[Hard mixed-query budget]\label{thm:ledger-log}
If $h\le n/2$ and $G+1\le nM\ell/4$, the success probability of a capped
policy on the presampled root law is at most $1/32$.
\end{theorem}
\begin{proof}
Process an unqueried output as one virtual ordinary selected point, without
allowing the policy to revise it.  If its true gradient has norm at most one,
Lemma~\ref{lem:hierarchy-analytic} forces every root port to one.  After the
repair loop, an unretired root child would still have an undisclosed parent
completion center and hence a flat-zero port at this point; if the port were
nonzero, its support test would have repaired the child.  Thus all $n$ leaves
and all $nM$ marked tasks have retired.

Normal rich retirements account for at most $hM$ tasks.  Every normal
ordinary retirement that is not early uses a disjoint window of $\ell$
ordinary slots.  All remaining tasks are counted in $E_{\rm early}$ or $Z$.
Consequently success implies the ledger \eqref{eq:road-ledger}.  Under the
query caps, $E_{\rm early}+Z\ge nM/4$.  Equations
\eqref{eq:repair-expectation} and \eqref{eq:early-expectation}, followed by
Markov's inequality, give
\[
 \PP\{\text{success}\}\le
 \frac{\E(E_{\rm early}+Z)}{nM/4}\le\frac1{32}.
\]
The comparison oracle returns the exact jet, so there is no additional
global flatness-failure event.
\end{proof}

\section{Dimension, Normalization, and the Main Theorem}
\label{sec:finish-log}
\begin{proof}[Proof of Theorem~\ref{thm:parameter}]
With $r_A=K_AM$, the numerical dimension in Section~\ref{sec:matrix} is
$D=\kappa r_AJ=O(M^{3/2})$.  There are $Mk_\bullet$ inner registers,
$8k_1$ height-one registers, and eight single-vector registers at every
higher height.  Hence
\begin{equation}\label{eq:physical-dimension-log}
 d_{\rm phys}=D+Mk_\bullet p_{\rm low}+8k_1p_{\rm low}
                         +8\sum_{q=2}^Lp_q.
\end{equation}
The level sum equals
\[
 p_{\rm low}\bigl(\min\{L,j_*-1\}-1\bigr)_+
 +8(L-j_*+1)_+.
\]
Since $p_{\rm low}=O(\log(2+M))$, $k_1=O(\log(2+M))$, and
$j_*=O(\log(2+M))$,
\[
 d_{\rm phys}\le C\{M^{3/2}+M\log(2+M)+\log^2(2+M)+L\}
 \le C_d\{M^{3/2}+\log(2+n)\}.
\]
No growing function of $n$ has been absorbed into a constant.

Lemma~\ref{lem:hierarchy-analytic} and
Theorem~\ref{thm:regularity-log} give gap $O(nM)$ and an absolute Hessian
Lipschitz constant $C_H$.  Fix
$b\ge\max\{1,\sqrt{C_H}\}$ and set $f(x)=bF(x/b)$.  Then
$\nabla f(x)=\nabla F(x/b)$, both query types transform one for one,
$\Lip(\nabla^2f)\le C_H/b^2\le1$, and the gap grows only by the absolute
factor $b$.  Pad unused coordinates with a public positive quadratic.
Theorem~\ref{thm:ledger-log} and $\ell\ge c\sqrt M$ exclude ordinary caps
below $c_0nM^{3/2}$, after absorbing the virtual output's one slot.
\end{proof}

\begin{proof}[Proof of Theorem~\ref{thm:main}]
Put $H=h+1$, $\mathcal L=\log(2+T)$, and
$K_T=\min\{d^{2/3},T/H\}$.  Choose a sufficiently small absolute
$\zeta>0$ and let $M=1+4096J^2$ be the largest admissible size not exceeding
$\zeta K_T$.  For sufficiently small $a$ and sufficiently large $T_*$,
this set is nonempty; adjacent admissible values have bounded ratio, so
\begin{equation}\label{eq:M-rounding}
 \zeta K_T/4\le M\le\zeta K_T.
\end{equation}
Choose $n=8^L$ as the largest power of eight not exceeding
$T/(C_EM)$.  If $\zeta\le1/(64C_E)$, then
\[
 \frac{T}{8C_EM}<n\le\frac{T}{C_EM},\qquad n>8H\ge2h,
 \qquad n\ge8.
\]
Thus the gap and rich-budget premises of Theorem~\ref{thm:parameter} hold.
Choose $\zeta$ so that $C_d\zeta^{3/2}\le1/4$.  The numerical dimension
term is then at most $d/4$.  Also $n\le T$, hence
$C_d\log(2+n)\le C_d\mathcal L\le d/4$ after choosing the absolute constant
in $d\ge C\mathcal L$ large enough.  The parameter theorem gives
\[
 Q_g^*(h;d,T)\ge c nM^{3/2}
 \ge c'T\sqrt M
 \ge c''\min\left\{Td^{1/3},\frac{T^{3/2}}{\sqrt{h+1}}\right\},
\]
using \eqref{eq:M-rounding}.  If an algorithm succeeded with probability
at least $2/3$ on every fixed legal input, its average success over the
presampled hard distribution would be at least $2/3$, contradicting
Theorem~\ref{thm:ledger-log} after the arbitrarily small computation-cutoff
loss.  This proves \eqref{eq:main}.
\end{proof}

\section{Discussion}\label{sec:scope}
The lower bound matches the gradient--Hessian upper tradeoff up to absolute
constants when $d\ge C\log(2+T)$ and $1\le h$, $h+1\le aT$.
The logarithmic threshold follows from the physical dimension estimate
$O(M^{3/2}+\log(2+n))$, while the initial gap is $O(nM)$.
Localizing the completion perturbations permits a fixed register dimension
at sufficiently high levels. The comparison argument charges accidental
discoveries by the number of tasks they retire; it does not require all
hidden parameters to remain undiscovered throughout the run.

\paragraph{Dimension and budget regimes.}
The theorem does not settle the tradeoff in dimension fixed independently
of $T$. At fixed $M$, this construction still uses eight separate vectors
at each of $\Theta(\log n)$ heights. \Cref{prop:fold} identifies an
obstruction to directly sharing the coordinates of two adjacent levels;
it is not a lower bound on the dimension required by every possible
construction. The parameter theorem allows $h=0$, whereas the main matching
statement is restricted to $h\ge1$. The transition to larger Hessian
budgets beyond $h+1\le aT$ is not resolved here.

\paragraph{Oracle conventions.}
Both query budgets are hard caps on every execution. The lower bound
allows measurable randomized adaptive queries and an arbitrary unqueried
output, with no precision, radius, span, or Hessian-norm restriction.
Expected query budgets constitute a different convention. The argument
uses a fixed input distribution independent of the algorithm's private
seed, and each comparison reply is the literal jet of that input.

\paragraph{Unnormalized units.}
For an objective $g$ with initial gap at most $\Delta>0$, Hessian
Lipschitz constant at most $L_2>0$, and target gradient norm $\eps>0$,
let
\[
 \alpha=\sqrt{\frac{\eps}{L_2}},\qquad
 f(x)=\frac{g(x_0+\alpha x)-g(x_0)}{\eps\alpha},
 \qquad T=\frac{\Delta\sqrt{L_2}}{\eps^{3/2}}.
\]
Then $\nabla f(x)=\nabla g(x_0+\alpha x)/\eps$ and
$\Lip(\nabla^2f)\le1$. Both types of queries transform one for one.
Consequently, in the corresponding range of $d,T,h$, the necessary
ordinary budget is
\[
 \Omega\!\left(\min\left\{
 \frac{\Delta\sqrt{L_2}\,d^{1/3}}{\eps^{3/2}},
 \frac{\Delta^{3/2}L_2^{3/4}}{\eps^{9/4}\sqrt{h+1}}
 \right\}\right).
\]

\clearpage
\appendix
\section{Translating the NALEN Upper Bound}\label{app:upper-translation}

For completeness, we justify the two adaptations of
Chen et~al.~\cite{CLLZ26} used in \Cref{tab:comparison}. This also
separates their weighted cost parameter $\bar d$ from ambient dimension
and accounts for the strict Hessian cap when $h=1$.

\begin{proposition}\label[proposition]{prop:upper-translation}
On the normalized class of \Cref{sec:model}, with $T\ge1$, deterministic
algorithms achieve the following worst-case query bounds:
\begin{enumerate}[label=(\roman*),leftmargin=2em]
\item For every integer $1\le h\le T$,
      $H\le h$ and $G=O(T^{3/2}/\sqrt h)$.
\item With no rich queries, $G=O(d+Td^{1/3})$.
\end{enumerate}
The constants are absolute and require no Hessian-norm bound.
\end{proposition}
\begin{proof}
Write $m$ for NALEN's integer Hessian refresh period, $N$ for its
iteration horizon, and $q=\lceil m^{1/3}\rceil$ for its averaging-block
length. The notation $q$ avoids a conflict with our normalized gap $T$.
Set its step-radius parameter to
\[
 D=\left(\frac{T}{Nm^{2/3}}\right)^{1/3},
 \qquad \eta=\frac{c_0}{mD},
\]
where $c_0>0$ is a sufficiently small absolute constant. The proof of
Theorem~4.1 in \cite[Appendices~B and D]{CLLZ26}, using the gap upper
bound $T$ in place of the actual gap, gives
\begin{equation}\label{eq:nalen-translation}
 \norm{\nabla f(\widehat y)}
 \le C_0\left(\frac{T}{DN}+m^{2/3}D^2\right)
 \le C_1\left(\frac{Tm^{1/3}}{N}\right)^{2/3}.
\end{equation}
Here $N\ge q$ suffices: the $K=\lfloor N/q\rfloor$ complete averaging
blocks contain $Kq\ge N/2$ iterations, so discarding a final incomplete
block changes only an absolute constant. At most $2N+K+1\le3N+1$
gradient calls suffice, including initialization and evaluation of all
$K$ candidate outputs. The number of distinct snapshot Hessians is
$\lceil N/m\rceil$. Our rich reply can supply the snapshot gradient
as well; ignoring this saving gives a valid upper bound on $G$.

For (i), choose
\[
 m=\left\lceil A(T/h)^{3/2}\right\rceil,\qquad N=hm,
\]
with an absolute constant $A\ge1$ large enough. There are exactly $h$
snapshot indices $0,m,\ldots,(h-1)m$. Since $q\le m\le N$,
\eqref{eq:nalen-translation} applies, and its right-hand side is at
most $C_1A^{-4/9}\le1$. Moreover,
$N\le A T^{3/2}/\sqrt h+h=O(T^{3/2}/\sqrt h)$ because $h\le T$.
This includes $h=1$ and does not infer a strict cap from a hidden
constant in an $O(1+T/m^{2/3})$ Hessian count.

For (ii), we first check that the proof tolerates an approximate
snapshot Hessian $B_j$ with
$\norm{B_j-\nabla^2f(z_j)}\le\delta:=mD$.
In the notation of \cite[Algorithm~1 and Appendix~B]{CLLZ26},
$y_{n+1}-z_n=(\Delta_{n+1}-\Delta_n)/2$.
Replacing the exact snapshot Hessian by $B_{\pi(n)}$ therefore adds
at most
\[
 \frac{\delta}{2}\norm{\Delta_{n+1}-\Delta_n}
\]
to the prediction error. The exact-Hessian bound on that error is
$((m+1)D/2)\norm{\Delta_{n+1}-\Delta_n}$. Thus the same estimate holds
with a larger absolute constant. Decreasing $c_0$ absorbs this change
in the regret estimate, and \eqref{eq:nalen-translation} remains valid.
The update uses a symmetric quadratic model, so the approximate
matrices below preserve the form of the subproblem.

To obtain such a matrix from ordinary queries at a snapshot $z$, let
$e_i$ be the coordinate vectors and form the columns
\[
 A_z e_i=\frac{\nabla f(z+s e_i)-\nabla f(z)}{s},
 \qquad s=\frac{2\delta}{\sqrt d},
 \qquad B_z=\frac{A_z+A_z^{\mathsf T}}2.
\]
The Lipschitz-Hessian Taylor bound gives an error at most $s/2$ in
each column, hence
$\norm{B_z-\nabla^2f(z)}\le\sqrt d\,s/2=\delta$.
This uses $d+1\le2d$ ordinary queries per snapshot, with no rich query.
Finally choose $m=d$ and
$N=\lceil A(d+Td^{1/3})\rceil$. Increasing $A$ if necessary makes
\eqref{eq:nalen-translation} at most one. Including all finite-difference
and output-check queries, the count is at most
\[
 3N+1+2d\lceil N/d\rceil=O(d+Td^{1/3}).
\]
All parameters and query caps are fixed from $T,d,h$ before the run.
If the initial gradient is already small, the algorithm terminates
immediately; otherwise its initialization is well-defined.
\end{proof}

\section{An Obstruction to Adjacent-Level Coordinate Sharing}\label{app:fold}
The theorem above still uses $O(\log n)$ separate physical height blocks.
The next proposition records an exact obstruction to identifying two
adjacent high-level arrays.  It is not an impossibility theorem for every
constant-dimensional encoding.

\begin{proposition}[Adjacent identification creates an uncertified stationary point]
\label[proposition]{prop:fold}
Consider the local objective $F_v^{(q+1)}$ at any node of height
$q+1$, with adjacent high levels $q,q+1\ge j_*$. Identify their corresponding
registers by $U_i^{(q)}=U_i^{(q+1)}=u_i$, $1\le i\le8$, leaving all lower
variables separate.  The restricted local objective has a stationary point at which
all eight parent completion ports are zero, for every choice of the relevant
hidden centers and descendants.
\end{proposition}
\begin{proof}
Write $R=R_q$, $w=w_q$; the parent scales are $6R,3w$.  Let $a_i$ be the
lower centers in the parent node's first child and $b_i$ the parent centers,
so $\norm{a_i}=R$ and $\norm{b_i}=6R$.  Inside the lower completion balls,
the lower self potential is
\[
 K_{a_i}(u)=P_R(a_i)-160w+\frac8w\norm{u-a_i}^2.
\]
Let $u_1^*$ minimize $K_{a_1}+G_{b_1}$, where $G_{b_1}$ uses width $3w$,
and for $2\le i\le8$ let $u_i^*$ minimize $K_{a_i}+P_{6R}$.  Each function
is strongly convex, so the minimizer is unique.  The nonquadratic gradient
has norm strictly less than $16$, while
$\nabla K_{a_i}=16(u-a_i)/w$; hence
$\norm{u_i^*-a_i}<w$.  All lower completion ports are therefore flat one.
But
\[
 \norm{u_i^*-b_i}\ge5R-w>9w,
\]
so every parent self perturbation and completion port is flat zero in a
neighborhood.

With the lower variables denoted by $Y$, the restricted scalar on this
neighborhood is
\[
 Q_{q-1}(Y)+\text{constant}
 +K_{a_1}(u_1)+G_{b_1}(u_1)
 +\sum_{i=2}^8\{K_{a_i}(u_i)+P_{6R}(u_i)\}.
\]
Identity \eqref{eq:loc-cancel} was used for the first parent block.  At
$Y=0$ and $u=u^*$ every displayed gradient vanishes, while every parent port
is zero.  The diagonal restriction preserves coercivity and changes the
Hessian-Lipschitz constant only by an absolute operator-norm factor.
\end{proof}

\end{document}